\documentclass[a4paper, reqno]{amsart} 
\usepackage{tikz} 
\usepackage{bm} 
\usepackage{listings} 
\usepackage{enumitem} 
\usepackage{wrapfig} 
\usepackage{subcaption} 
\usepackage{float} 
\usepackage{amssymb} 
\usepackage{amsmath, mathtools}

\usetikzlibrary{intersections}
\usetikzlibrary{decorations.markings}

\usepackage{xcolor}
\usepackage[all]{xy}
\usepackage{hyperref}
\usepackage[T1]{fontenc}
\usepackage{todonotes}
\usepackage{enumitem}

\newtheorem{theorem}{Theorem}[section]
\newtheorem{proposition}[theorem]{Proposition}

\newtheorem{corollary}[theorem]{Corollary}

\theoremstyle{definition}
\newtheorem{definition}[theorem]{Definition}
\newtheorem{remark}[theorem]{Remark}

\newtheorem{example}[theorem]{Example}

\newcommand{\lra}{\longrightarrow}

\renewcommand{\H}{{\mathcal{H}}}

\renewcommand{\sup}{\mathrm{sup}}
\renewcommand{\inf}{\mathrm{inf}}

\newcommand{\R}{{\mathbb{R}}}
\newcommand{\F}{{\mathbb{F}}}

\newcommand{\Z}{{\mathbb{Z}}}
\newcommand{\SV}{{\mathrm{SV}}}
\newcommand{\slk}{{\mathrm{lk}}}
\newcommand{\bfp}{{\mathbf{p}}}
\DeclareMathOperator{\Hom}{Hom}

\DeclareMathOperator{\lk}{Lk}

\newcommand{\emb}{\ensuremath{\mathrm{emb}}}
\DeclareMathOperator{\Ker}{Ker}

\begin{document}

\author{Jelena Grbi\'c${}^{\dag}$}
\address{School of Mathematical Sciences\\
University of Southampton\\ Southampton, UK}
\email{j.grbic@soton.ac.uk}

\author{Jie Wu${}^{\dag,\ddag}$}
\address{School of Mathematical Sciences\\
Center of Topology and Geometry based Technology\\
Hebei Normal University\\
No 20 Road East, 2nd Ring South\\
Yuhua District, Shijiazhuang\\
Hebei, 050024 CHINA}
\email{wujie@hebtu.edu.cn}

\author{Kelin Xia${}^{\dag}$}
\address{School of Physical and Mathematical Sciences\\
Nanyang Technological University\\
SPMS-MAS-05-18, 21 Nanyang Link, 1\\
Singapore 63737}
\email{xiakelin@ntu.edu.sg}

\author{Guowei Wei${}^{\dag}$}
\address{Department  of  Mathematics\\
Michigan State University\\
D301 Wells Hall, 619 Red Cedar Road\\
East Lansing, MI 48824, USA}
\email{weig@msu.edu}

\title[Topological approach to Data  Science]{A Unified Topological Approach to Data Science\footnote{\hspace*{-.4cm}$\dag$ first authors\\
$\ddag$ corresponding author}}

\begin{abstract}
We establish a new theory which gives a unified topological approach to data science, by being applicable both to point cloud data and to graph data, including networks beyond pairwise interactions. We generalize simplicial complexes and hypergraphs to super-hypergraphs and establish super-hypergraph homology as an extension of simplicial homology. Driven by applications, we also introduce super-persistent homology. 
\end{abstract}

\subjclass[2010]{}
\keywords{}

\maketitle

\section{Introduction}

Topological data analysis (TDA) is a fast growing area of research stemming from work on persistent homology such as ~\cite{ELZ,ZC} and the pioneering paper of Carlsson~\cite{Carlsson}. TDA has been successfully applied in various areas of the sciences and technology such as material science~\cite{KGKM, KLTSXPSM,Nakamura}, 3D shape analysis~\cite{Skraba, Turner}, multivariate time series analysis~\cite{Seversky}, molecular biology~\cite{Wei2018,Wei2017,Wei2017-2,XW}, sensor networks~\cite{Silva}, scientific visualization~\cite{Tierny2017}, machine learning~\cite{Chazal, Pokorny}, etc. The wide applications of TDA have made topology as one of the most commonly used mathematical tools in Data Science~\cite{Data}.
 As the applications of TDA continue to expand, the subject has inspired new theories in topology that will enable its further applications to engineering, the natural and social sciences, and the arts.



In the survey paper~\cite{CM},  Chazal and Michel outlined  a pipeline that stresses the role of topology and geometry in data science:
\begin{enumerate}
\item[(i)] Input data is given in the form of a finite set of points coming with a notion of distance.
\item[(ii)] A ``continuous shape'' is built from the input data: this results in a structure over the data.
\item[(iii)] Topological and geometric information is extracted from the structure.
\item[(iv)] The topological and geometric information is the output of the analysis and forms the new representation of the data, allowing for an in-depth modeling
of the original data.
\end{enumerate}
This approach can be naturally applied to point cloud data with a drawback that it can not be immediately or directly applied to non-Euclidean data such as abstract relationships within graphs. In this setting, data is always treated as a geometric space embedded in an Eucleadian space and must come with the notion of distance. 


The purpose of this article is to provide a unified topological approach to data science that is suitable for both point cloud data and graph data. We depart from the requirement that data needs to come with a notion of metric and instead equip it with the geometry-free notion of scoring schemes. In our setting, we explore topological structures on graph data with scoring schemes.  This approach also recovers the current methods in TDA for analyzing point cloud data, in particular popular persistent homology, by associating the complete graph on the point cloud and specifying a particular scoring scheme.

We start with a graph, which is the working graph for the data analytic purpose. Our approach consists of the following steps:

\begin{enumerate}
\item[(A)] We introduce a homology theory of a collection of subgraphs of the working graph, which is a generalisation of simplicial homology theory. This homology theory gives ``topological invariants'' for collections of subgraphs associated to data.
\item[(B)] On the working graph $G$ we assign a \textit{scoring scheme}, a function from the set of subgraphs of $G$ to the set of real numbers. The scoring scheme induces persistence on homologies in (A), which we call \textit{super persistent homology}, as well as its derived \textit{topological features} such as super-persistence diagrams and super-persistence modules.
\item[(C)] The current persistent homology of point cloud data can be deduced from (A) and (B). Hence our approach is suitable for performing topological data analysis on both graphic data and point cloud data.
\end{enumerate}
The pineline of our super persistent homology is as follows:
\begin{enumerate}
    \item The input is assumed to be a finite (or infinite) graph $G$ with
    \begin{enumerate}
        \item [(i)] a scoring scheme and
        \item [(ii)] a selection of subgraphs.
\end{enumerate}
The definition of the scoring scheme on the data is usually given as an input or guided by applications. It is however important to notice that the choice of a scoring scheme may be critical to revealing interesting topological and geometric features of the data. The selection of subgraphs on the data is also usually given as an input or guided by the applications at hand. Again it is important to notice that the selection of subgraphs may be critical to revealing interesting topological and geometric features of the data.
\item An abstract ``geometry-like'' shape is built on top of the data in order to detect its underlying topological structure. This is a nested family of super-hypergraphs filtered by the scoring scheme that reflects the structure of the data at different scales. Super-hypergraphs can be seen as higher dimensional generalizations of neighboring graphs that are classically built on top of data in many standard data analysis or learning algorithms. The challenge here is to define structures that reflect relevant information about the data and that can be effectively constructed and manipulated in practice.
\item The extracted topological information provides new families of features and descriptors of the data. These can be used to better understand the data or to be combined with other features for further analysis and machine learning tasks. An important gain in this step is the demonstration of the added-value and the complementarity, with respect to other features, of the information obtained by super persistent homology.
\item Adjust the choice of scoring scheme and the selection of subgraphs to get better features and descriptors of the data.
\item The procedure can be iterated on the choices of scoring scheme and selections of subgraphs to obtain the best features and descriptors of the data.
\end{enumerate}

An important point of the above pipeline is that it provides an indeterministic  approach, which avoids \textit{topological noise}. According to~\cite[page 2]{CM}, the topological noise is noise created by deterministic method due to the fact that deterministic approaches do not take into account the random nature of data.

In the pursuit of a unified approach that will drive the development and applications of topology forward, we give the further motivating observations:
\begin{enumerate}
\item[($1$)] As the scale of information being processed increases, there is a need to understand the global picture. This requires looking at relationships of~higher order than the pairwise relationship currently considered by graphs. 

\item[($2$)] Many applications use simplicial complexes as core modeling objects which allow for the application of tools from algebraic topology. These work particularly well for point cloud data as simplices are uniquely determined by their vertices. However, in the study of graph data a more general concept is needed to build a topological model of the data as two different subgraphs might be defined on the same set of vertices, see Figure~\ref{fig:exam_2}. 

	\begin{figure}[h]
			\centering
			\begin{minipage}{0.4\textwidth}
				\centering
				\begin{tikzpicture} [scale=0.7, inner sep=2mm]
				\coordinate (1) at (0,0);
				\coordinate (2) at (2,0);
				\coordinate (3) at (2,2);
				\coordinate (4) at (0,2);
				\draw (1) node[below left] {$v_1$}-- (2) node[below right] {$v_2$} -- (3)  node[above right] {$v_3$} -- (4) node[above left] {$v_4$} -- (1) ;
				\foreach \i in {1,...,4} {\fill (\i) circle (1.8pt);}; 
				\end{tikzpicture} 
			\end{minipage}\qquad
			\begin{minipage}{0.4\textwidth}
				\centering
				\begin{tikzpicture} [scale=.7, inner sep=2mm]
				\coordinate (1) at (0,0);
				\coordinate (2) at (2,0);
				\coordinate (3) at (2,2);
				\coordinate (4) at (0,2);
				\draw (1) node[below left] {$v_1$}-- (2);
				\draw (1) node[below left] {$v_1$}-- (3);
				\draw (1) node[below left] {$v_1$}-- (4); 
				\foreach \i in {1,...,4} {\fill (\i) circle (1.8pt);}; 
				\end{tikzpicture} 
			\end{minipage}
			\caption{Two different subgraphs on the same vertex set of the complete graph on 4 vertices.}
				\label{fig:exam_2}
		\end{figure}
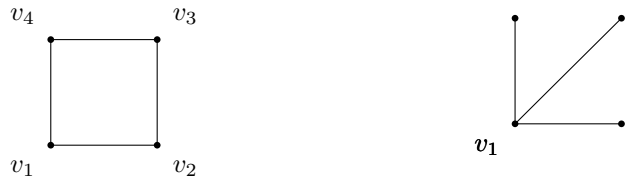

\item[($3$)] To detect higher-order interactions, such as those appearing in social systems and gene pathways, which may not form a closed system as required by a simplicial complex, the notion of hypergraphs can be used. See for example a modeling of the collaboration network in Figure~\ref{collaboration}. This now requires development of topological theories of hypergraphs. 
\begin{figure}[h]
			\centering
			
			\begin{minipage}{0.4\textwidth}
				\centering
				\begin{tikzpicture} [scale=0.7, inner sep=2mm]
				\coordinate (1) at (0,0);
				\coordinate (2) at (2,0);
				\coordinate (3) at (1,2);
				\draw (1) node[below left]{$A$}-- (2) node[below right] {$B$} -- (3)  node[right] {$C$}-- (1);
				\foreach \i in {1, 2, 3} {\fill (\i) circle (2pt);} 
			    \draw[thick, dashed](0.1, 0) -- (1.95, 0);
			    \draw[thick, dashed](0.05,0.1)--(0.95,1.9);
			    \draw[thick, dashed](2, 0.05)--(1.05,1.9);
				\fill[lightgray, fill opacity=0.4] (1)--(2)--(3)--(1);
				\end{tikzpicture} 
			\end{minipage}
			\caption{A collaboration network in which three researchers have a joint paper but no two have a paper together}
			\label{collaboration}
		\end{figure}
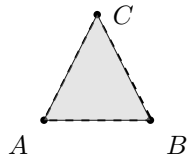

\item[($4$)] To understand the common nature of point cloud and graph data, there is a need to establish a unified approach to applications of topology. That will allow the introduction of new topological structures and thereby broaden the ranges of problems that topology can resolve. 
\end{enumerate}



There have been extensive explorations of topological and categorical structures on graphs. In Section~\ref{sec:topstructures} we survey the existing topological structures associated to a graph. The geometric realization, for example, of the neighborhood complex of a graph is quite different from that of the clique complex. This indicates that there are different topological structures on the same graph. Our introduction of scoring schemes gives a new way of assigning a wide range of topological structures on the space of subgraphs. In applications, this enables various topological methods to analyze a working graph given by the data. These can then extract different aspects of the data and can characterize the behaviour of higher order relationships.

An approach to investigating higher-order interactions in the graph data is to consider a collections of subgraphs of the working graph. This gives rise to the following mathematical question: Given $\mathcal{H}$, a collection of subgraphs of a working graph $G$, how can topological structures be introduced on $\mathcal{H}$ with as few constraints as possible?

In this paper, we will answer this question, both from a mathematical and applications in data science perspective. 

The central focus of this paper is the introduction of super-hypergraphs that resolves the limitations of completeness and vertex determination in simplicial complexes. In algebraic topology, $\Delta$-sets resolve the problem of vertex determination while generalising simplicial complexes. A $\Delta$-set $X$ is a graded set $X_0, X_1,X_2,\cdots$, where $X_n$ consists of $n$-dimensional simplices of $X$ endowed with face operations $d_i\colon X_n \longrightarrow X_{n-1}$, $0 \leq i \leq n$, satisfying the $\Delta$-identity $d_id_j=d_jd_{i+1}$ for $i\geq j$.  A $\Delta$-set can be described in terms of feed-forward neural networks, see Subsection~\ref{subsection4.5}. Using $\Delta$-sets, we can model a collection of two or more subgraphs sharing the same vertices.

To ensure that topological structures are built on a collection  $\H$ of subgraphs with as few constrains as possible, we introduce a super-hypergraph (see Definition~\ref{def:sup_hyp_graph}) as a graded subset
of a $\Delta$-set. If a $\Delta$-set is given by an oriented simplicial complex, then a super-hypergraph coincides with a hypergraph. Therefore, a super-hypergraph is an extension of a hypergraph that allows hyperedges to form a multiset. An important aspect in the present paper is that simplicial homology can be naturally extended to a homology theory of super-hypergraphs as described in Section~\ref{homology}, giving new topological invariants of super-hypergraphs.

 

By embedding hypergraphs data in $\Delta$-sets, our approach allows for observing different topological aspects of the data, the reduction of topological noise and constructing relevant confidence regions on the topological aspects of the data. Notably, varying the $\Delta$-set makes possible to depart from the current deterministic study of hypergraph models. 

To bring our topological theory to applications, we introduce persistent super-hypergraph homology and interpret it as persistence on graph data. This uses scoring schemes on a space of finite subgraphs of the working graph. The peristent super-hypergraph homology generlaizes the classical persistent homology and in Subsection~\ref{ordinary persistent homology} we illustrate that relationship.

Super-persistent homology has the desired stability property~\cite{RW}, which gives robustness to this approach allowing it to be used in various applications. Hence super-persistent homology is a novel topological approach that can be applied to broader objects in data science.

This paper is a theoretic research resulting in a framework which establishes a unified topological study of data science. In Section~\ref{sec:potent_apps}, we outline how this approach could be applied to work in areas like bio-molecular science, drug design and networks with group interactions.

\tableofcontents

\section{Topological structures associated to graphs}\label{sec:topstructures}

Throughout mathematics numerous topological and categorical structures on graphs have been explored. In this section, we will survey various simplicial complexes associated to graphs that allow us to consider the space of subgraphs of a given graph from a topological perspective.

A \textit{directed (multi-)graph} (or \textit{multi-digraph} or \textit{quiver}) is a pair $G=(V(G),E(G))$ together with a function
$
\mathrm{end}\colon E(G)\longrightarrow V(G)\times V(G)$ given by 
$$ 
\mathrm{end_G}(e)\mapsto (i(e),t(e))
$$
where $V(G)$ is the \textit{vertex} set, $E(G)$ is the \textit{edge} set, $i(e)$ is the \textit{initial} vertex of the edge $e$, and $t(e)$ is the \textit{terminal} vertex of $e$.

An \textit{undirected (multi-)graph}\footnote{We follow the definition of a multi-graph in~\cite{Diestel,Balakrishnan}. In some literature such as~\cite{Rahman}, a multi-graph is defined by requiring the edge set to be a multi-set. The difference is that the edges between two vertices are labeled by $E(G)$ together with the incidence map $\mathrm{end}_G$. Such a definition coincides with the definition on quiver (as directed multi-graph)~\cite{Schiffler}.} is a pair $G=(V(G),E(G))$ together with a function
$
\mathrm{end_G}\colon E(G)\longrightarrow (V(G)\times V(G))/\Sigma_2$  given by 
$$
\mathrm{end_G}(e)\mapsto \{i(e),t(e)\}
$$
where $(V(G)\times V(G))/\Sigma_2$ is the orbit set of $(V(G)\times V(G))$ modulo the $\Sigma_2$-action given by permuting the coordinates, $V(G)$ is the \textit{vertex} set, $E(G)$ is the \textit{edge} set, $\mathrm{end_G}$ is an \textit{incidence relation} that associates with each edge of $G$ an unordered pair of, possibly equal, elements of $V(G)$. In this definition of a directed/undirected (multi-)graph, the empty graph is allowed.

A \textit{subgraph} $H$ of a directed/undirected (multi-)graph $G$ is a graph $H=(V(H),E(H))$ with $V(H)\subseteq V(G)$, $E(H)\subseteq E(G)$ and $\mathrm{end}_H=\mathrm{end}_G|_{E(H)}$.

A directed/undirected graph $G$ is  \textit{simple} if $\mathrm{end}_G$ is injective and the image $\mathrm{end}_G(E(G))$ is disjoint from the diagonal $\Delta(V(G))$ in $V(G)\times V(G)$ or $(V(G)\times V(G))/\Sigma_2$. This means that there are no loops or multi-edges between two vertices.

From the perspective of applications, the initial data is represented by a given graph $G$ and let $\H$ be a collection of subgraphs of $G$. Our goal is to investigate the possible topological structures on $\H$. However, before we address this general question, we review some classical constructions of simplicial complexes associated to graphs.

\subsection{Clique Complexes}
Typically, the study of collections of subgraphs has focused on measuring how strongly connected different parts of a graph are. A clique (or flag) complex and an independence complex (the clique complex of the complementary graph) are topological spaces that contain information about the connectivity of a graph. These are widely used objects in mathematics and its applications, see ~\cite{Aharoni-Berger-Ziv,Barmak,Ehrenborg-Hetyei,Engstrom,Kozlov99} for some recent works.

A \textit{complete graph} is a simple graph $G = (V(G), E(G))$ with the property that every pair of distinct vertices of $G$ are adjacent in $G$.

A \textit{clique} of a graph $G$ is a  complete subgraph of $G$.
 
The \textit{clique complex} of a simple graph $G$ is the abstract simplicial complex $\mathrm{Clique}(G)$ whose simplices consist of all cliques of $G$. An $n$-simplex $\sigma$ in $\mathrm{Clique}(G)$ is a clique of $G$ with $(n+1)$ vertices, and a face of a simplex $\sigma\in \mathrm{Clique}(G)$ is a complete subgraph obtained by deleting some vertices of $\sigma$.

When working with a multi-graph $G=(V,E)$, the set of cliques $\mathrm{Clique}(G)$ is generally not a simplicial complex as this requires that all simplices are uniquely determined by their vertex set. For example, let $G$ be a multi-graph with two vertices $v$ and $w$ and two edges $e_1$ and $e_2$ joining them. Then
$$
\mathrm{Clique}(G)=\{\bar e_1,\bar e_2, v, w\}
$$
has two $1$-simplices $\bar e_1$ and $\bar e_2$ sharing the same vertices $v$ and $w$, see Figure~\ref{figure_multigraph}. Therefore, a more suitable object for describing the topological structure of $\mathrm{Clique}(G)$ is a $\Delta$-set.

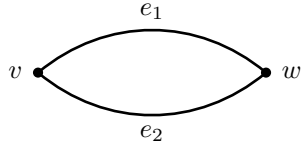
\begin{figure}[h]
			\centering
			\begin{minipage}{0.4\textwidth}
				\centering
	\begin{tikzpicture} [scale=1, inner sep=2mm]
			\coordinate (a) at (0.5,0.5);
			\coordinate (b) at (0.5,-0.5);
			\coordinate (e) at (2,0); 
			\coordinate (f) at (-1,0); 
			\draw[line width=1pt] (f) to [out=-40,in=-140] (e); 
			\draw[line width=1pt] (f) to [out=40,in=140] (e); 
			\foreach \i in {e, f} {\fill(\i) circle (1.9pt);} 
		    \draw (a) node[above] {$e_1$};
		    \draw (b) node[below] {$e_2$};
			\draw (e) node[right] {$w$};
			\draw (f) node[left] {$v$};
			\end{tikzpicture} 
				\end{minipage}
			\caption{The multi-graph $G$, which looks the same as the clique complex  $\mathrm{Clique}(G)$.}
			\label{figure_multigraph}
		\end{figure}
\begin{definition}\label{def:delta_set}~\cite{Curtis, Wu}
    A \textit{$\Delta$-set} $X_*$ is a sequence of sets $X_*=(X_n)_{n\geq0}$ with maps $d_i\colon X_n\to X_{n-1}$, for $0\leq i\leq n$ and  $n\geq 1$,  called \textit{face operations}, satisfying the following $\Delta$-identity
    \begin{equation}\label{equ:delta_set}
        d_id_j=d_jd_{i+1}\textrm{ for } i\geq j.
    \end{equation}
\end{definition}

For an undirected multi-graph $G$, the $\Delta$-set structure on $\mathrm{Clique}(G)$ is given in the following way. 
Assign a total ordering to $V(G)$ and define $\mathrm{Clique}_n(G)$ to be the set of cliques of $G$ that have exactly $n+1$ vertices. For $\sigma\in \mathrm{Clique}_n(G)$ with vertices $v_0<v_1<\cdots<v_n$, define $d_i\sigma =\sigma - v_i$, the subclique of $\sigma$ obtained by deleting the vertex $v_i$ and the edges incident to $v_i$ for $0\leq i\leq n$. It is straightforward to check that $\mathrm{Clique}_*(G)$ forms a $\Delta$-set\footnote{The definition of the $\Delta$-set $\mathrm{Clique}_*(G)$ depends on the given order on vertices of $G$, but the homology of $\mathrm{Clique}_*(G)$ is independent on this choice because the geometric realization of a $\Delta$-set is a $\Delta$-complex~\cite[Proposition 1.39, p. 51]{Wu} in the sense of Hatcher~\cite{Hatcher}.}.

\subsection{Neighborhood Complexes and Jonsson's Graph Complexes}

We proceed by considering a collection of simplicial complexes associated to graphs which will naturally lead to new constructions suitable for studying spaces of subgraphs. We start with a famous construction of the neighborhood complex of a graph. This was introduced by Lov\'asz~\cite{Lovasz} in 1978 in his work on Kneser's conjecture which laid the foundations of topological combinatorial by introducing homotopy theoretical methods to combinatorics. Nowadays, the research area of topological combinatorics is very active and fruitful.  The generalization by Lov\'asz of the neighborhood complexes  to the Hom complex ~\cite{Babson-Kozlov2006, Kozlov}, which has the same homotopy type as the clique complex of an exponential graph~\cite[Remark 3.6]{Dochtermann}, was used in a breakthrough work of Babson and Kozlov~\cite{Babson-Kozlov} to solve the Lov\'asz conjecture which relates the chromatic number of a graph with the homology of its Hom complex. Our theory  is based on the exploration of the interplay between topology and combinatorics.

The \textit{neighborhood complex} $\mathcal{N}(G)$ of a graph $G$ is a simplicial complex on vertex set $V(G)$ in which an $n$-simplex is a subset of $V(G)$ with $n+1$ vertices such that all vertices are adjacent to an other vertex in $G$.

As we discussed in the previous subsection, $\mathrm{Clique}(G)$ may not be a simplicial complex for a multi-graph $G$. However, for any graph $G$ the neighborhood complex $\mathcal{N}(G)$ is a simplicial complex.

The topology on the geometric realization of $\mathcal{N}(G)$ can be quite different from that of $\mathrm{Clique}(G)$ in general. For example, let $G$ be a graph with three vertices $a,b,c$ and two edges given by $ab$ and $bc$. Then $\mathcal{N}(G)=\{\{a,c\},\{a\},\{b\},\{c\}\}$, which is not connected, see Figure~\ref{fig:exam_3_Nbhd(g)}, and $\mathrm{Clique}(G)=\{\{a,b\},\{b,c\},\{a\},\{b\},\{c\}\}$ which is connected, see Figure~\ref{fig:exam_3_graph_Clique}. This indicates that there are various topological structures one could construct for a given working graph $G$.
\begin{figure}[h]
			\centering
			\begin{minipage}{0.4\textwidth}
				\centering
				\begin{tikzpicture} [scale=1.5, inner sep=2mm]
				\coordinate (1) at (0,0);
				\coordinate (2) at (0.5,1);
				\coordinate (3) at (1,0);
				\draw (1) node[below left] {$a$}-- (2) node[above] {$b$} -- (3)  node[below right] {$c$};
				\foreach \i in {1,...,3} {\fill (\i) circle (1.8pt);}; 
				\end{tikzpicture} 
				\subcaption{The graph $G$, which is the same as $\mathrm{Clique}(G)$.}
				\label{fig:exam_3_graph_Clique}
			\end{minipage}\qquad
			\begin{minipage}{0.4\textwidth}
				\centering
				\begin{tikzpicture} [scale=1.5, inner sep=2mm]
				\coordinate (1) at (0,0);
				\coordinate (2) at (0.5,1);
				\coordinate (3) at (1,0);
				\draw (1) node[below left] {$a$}-- (3)  node[below right] {$c$};
				\draw (2) node[below right] {$b$};
				\foreach \i in {1,...,3} {\fill (\i) circle (1.8pt);}; 
				\end{tikzpicture} 
				\subcaption{The neighborhood complex of $G$, $\mathcal{N}(G)$.}
				\label{fig:exam_3_Nbhd(g)}
			\end{minipage}
			\caption{}
		\end{figure}

In~\cite[p.26]{Jonsson}, Jonsson defines a graph complex in the following way. A \textit{graph complex}\footnote{Kontsevich also introduced graph complexes with a different defintion~\cite{Kontsevich}.} on a finite vertex set $V$ is a family $\mathcal{E}$ of simple graphs on the vertex set $V$ such that $\mathcal{E}$ is closed under deletion of edges; if $H \in\mathcal{E}$ and $e\in H$, then $H-e \in \mathcal{E}$. Identifying $H= (V,E) \in \mathcal{E}$ with the edge set $E$, we may interpret $\mathcal{E}$ as a simplicial complex. There are potentially different graph complexes on a given vertex set $V$ because the collection of simple graphs can be chosen in a different way.

With a slight modification to Jonsson's definition, namely adding a hypothesis that the simple graphs in $\mathcal{E}$ are subgraphs of $G$, we retain the central ideas of Jonsson's construction but also gain control over the space of subgraphs. In contrast to clique complexes and neighborhood complexes, the face operations in Jonsson's graph complex are given by deleting edges. Also, the construction of a graph complex is not fully determined by $G$ as there are various choices for families $\mathcal{E}$ of simple subgraphs of $G$ that can form graph complexes. A non-deterministic characteristic of these complexes might be useful in data science as the family $\mathcal{E}$ can be adjusted for each iteration of the analysis.

\subsection{Path Complexes}
Considering hypergraphs as a combinatorial generalization of simplicial complexes allows the construction of the path complex of a given digraph.

The topological exploration of path complexes was first introduced by Shing-Tung Yau and his collaborators in a series of papers ~\cite{GLMY-2,GLMY-1,GLMY0,GLMY1,GLMY2,GLMY3,GLMY3-5,GLMY4,GLMY5,GLMY6}. Motivated by ideas from physical applications, A. Dimakis and F. M\"uller-Hoissen attempted to construct the cohomology of digraphs~\cite{Dimakis-Muller-Hoissen1,Dimakis-Muller-Hoissen2}. They considered path complexes on an intuitive level without a precise definition of the corresponding cochain complex.

In this subsection, we survey the main ideas of path complexes of simple digraphs.

Let $G$ be a simple digraph. A \textit{directed path} in $G$ is an alternating sequence $\lambda=v_0\alpha_1v_1\alpha_2v_2\cdots\alpha_kv_k$, with all vertices $v_i$ distinct for $0\leq i\leq k$ and the edges, $\alpha_i$, are incident out of $v_{i-1}$ and incident into $v_i$ for $1\leq i\leq k$.

Let $\mathcal{P}$ be the set of directed paths in $G$. We want to associate a combinatorial object to $G$ built out of directed paths. Since $G$ is simple, there is at most one edge joining two distinct vertices. So a directed path $\lambda=v_0\alpha_1v_1\alpha_2v_2\cdots\alpha_kv_k$ is determined by its vertices $v_0,v_1,\ldots, v_k$. Thus we consider $\lambda=v_0\alpha_1v_1\alpha_2v_2\cdots\alpha_kv_k$ as an abstract $k$-simplex $\{v_0,v_1,\ldots,v_k\}$. For $\mathcal{P}$ to be a simplicial complex, any nonempty subset of $\{v_0,v_1,\ldots,v_k\}$ must be a simplex. In other words, any subsequence $(v_{i_0}, v_{i_1}, v_{i_2},\ldots,v_{i_t})$, $0\leq i_0<i_2<\cdots<i_t\leq k$ of $\lambda$ must forms a directed path in $G$. This is not true in general. For example, if $v_0\alpha_1v_1\alpha_2v_2$ is a directed path in $G$, then there may not exist an edge incident out of $v_0$ and incident into $v_2$ in $G$, that is, $(v_0,v_2)$ may not form a directed path. Therefore, a structure to consider on the set $\mathcal{P}$ is that of a hypergraph.

\begin{definition}\label{def:hypergraph}
    A \textit{hypergraph} $\H$ is a pair $\H=(V_\H,\mathcal{E}_\H)$, where the vertex set $V_\H$ is a finite or infinite set and the hyperedge set $\mathcal{E}_\H$ is a collection of finite nonempty subsets of  $V_\H$.
\end{definition}

The set $\mathcal{P}$ becomes a hypergraph with its vertex $V(G)$ and the hyperedge set given by directed paths in $G$. By definition, an abstract simplicial complex is a hypergraph with the additional condition that any nonempty subset of a hyperedge is a hyperedge. Therefore, a hypergraph can be viewed as a simplicial complex with some faces missing, where a hyperedge of cardinal $k+1$ is a $k$-simplex in the terminology of simplicial complexes. The approaches in~\cite{Dimakis-Muller-Hoissen1,Dimakis-Muller-Hoissen2} and Yau's school lead to the embedded homology of hypergraphs as an extension of simplicial homology theory as introduced in~\cite{BLRW}.

By allowing vertex repetition in directed paths, we get directed walks. The walk complex $\mathcal{W}(G)$ for a digraph or quiver (i.e. directed multi-graph) $G$, is similar to the path complex but with we replace directed paths with directed walks. Therefore, $\mathcal{W}(G)$ is an extension of the notion of the nerve of a category in the following sense. Consider a category $\mathcal{C}$ as a quiver with the composition operation on head-to-tail arrows. Then the nerve of category $\mathcal{C}$ is the walk complex of quiver $\mathcal{C}$.

\subsection{Introducing new topological structures on collections of subgraphs}\label{subsec:new_top_structures}
 In our novel approach, we seek to provide a natural extension of simplicial complexes and hypergraphs and their homologies to allow us to understand the structures of the various spaces related to graphs from a topological perspective. In doing so we will be constructing topology on spaces that look like ``partial'' $\Delta$-sets. 

Let $G$ be a directed/undirected (multi-)graph, and let $\H$ be a collection of subgraphs of $G$. Before looking at possible topological structures on $\H$, let us first recall how abstract simplical complexes and $\Delta$-sets can be realized as topological spaces.

Recall that an abstract simplicial complex is a $\Delta$-set under a choice of total order on its vertices. Let $X_*$ be a $\Delta$-set. Then each element in $X_n$ represents an abstract $n$-simplex. In the geometric realization $|X_*|$, one assigns a standard geometric $n$-simplex to each element $x\in X_n$ and the face operation $d_i\colon X_n\to X_{n-1}$, $x\mapsto d_i(x)$, induces a corresponding gluing of the geometric $(n-1)$-simplex labeled by $d_i(x)$ into the $i$-th face of the geometric $n$-simplex labeled by $x$. (See~\cite{Wu} for details.)

We need two pieces of information about $\H$ to establish a topological structure. The first, is a partition on the set $\H$ by a function $f\colon \H\longrightarrow \mathbb{N}=\{0,1,2,\ldots\}$. The subset $\H_n=f^{-1}(n)$ will be viewed as the set of $n$-simplices. The second, is a definition of face operations $d_i\colon \H_n\to \H_{n-1}$, $0\leq i\leq n$, satisfying the $\Delta$-identity ~(\ref{equ:delta_set}). This is important for establishing a topological structure as it plays the role for gluing corresponding geometric simplices together.

The path complex as an example of a subspace of subgraphs suggests that face operations $d_i\colon \H_n\to \H_{n-1}$ may be only partially defined on spaces of subgraphs. However, we still wish to have a homology theory of these objects. This leads to the following definitions.

\begin{definition}\label{def:sup_hyp_graph}
    A \textit{super-hypergraph} is a pair $(\H, X)$, where $X$ is a $\Delta$-set and $\H$ is a graded subset of $X$. We call $\H$ a \textit{super-hypergraph born from $X$}, and $X$ is called \textit{a parental $\Delta$-set} of $\H$. The \textit{$\Delta$-closure} of $\H$ in $X$ is defined by
    $$
    \Delta^X(\H)=\bigcap\{ Y \ | \ \H\subseteq Y \textrm{ as a graded subset and } Y\subseteq X \textrm{ as a $\Delta$-subset}\}.
    $$
\end{definition}

\begin{definition}\label{def:sup_hyp_graph_morphisms}
A \textit{morphism} $\phi\colon (\H, X)\lra (\H', Y)$ of super-hypergraphs is a $\Delta$-map $\phi\colon X\lra Y$ such that $\phi(\H)\subseteq \H'$.
\end{definition}

In the next sections, we will develop a homology theory of super-hypergraphs despite the fact that there are only partially defined face operations.  In the next two subsections, we consider some natural face operations to create topological structures.

\subsection{Vertex-deletion topology}
Let $G$ be a directed/undirected (multi-)graph. Let $\H$ be a collection of finite subgraphs of $G$. Assign to $\H$ the grading function $f_v\colon \H\longrightarrow \mathbb{N}=\{0,1,2,\ldots\}$ given by the size, namely, for $H\in\H$, let $f_v(H)=|V(H)|-1$. Let $\H_n=f^{-1}_v(n)$. Next step is to define face operations to obtain a topological structure. There are several natural approaches available.

\subsubsection{Primary vertex-deletion topology} Assume that the vertex set $V(G)$ is totally ordered. A geometric way to define face operations is to delete a vertex together with all edges incident to this vertex. More precisely, let $H\in \H_n$ with vertices $v_0,v_1,\ldots,v_n$. Define $d_i(H)$ for $0\leq i\leq n$ to be the subgraph of $H$ by deleting $v_i$ together with any edges joining with $v_i$. This vertex deletion does not ensure that $d_i(H)$ lies in $\H_{n-1}$. Let
\begin{equation}\label{equation2.2}
\Delta(\H)=\{d_{i_1}d_{i_2}\cdots d_{i_t}(H) \ | \ H\in \H, \quad 0\leq i_1<i_2<\cdots<i_t\leq |V(H)|-1\}
\end{equation}
be the family of subgraphs of $G$ obtained from $\H$ together with iterated faces on the subgraphs in $\H$. It is straightforward to check that $\Delta(\H)$ is a $\Delta$-set, and $\H\subseteq \Delta(\H)$ is a graded subset\footnote{From the $\Delta$-identity~(\ref{def:delta_set}), $\Delta(H)$ contains all iterated faces on the subgraphs in $\H$, which is the smallest family of subgraphs of $G$ containing $\H$ that is closed under the face operation.}. Hence $(\H, \Delta(\H))$ is a super-hypergraph.

\begin{definition}\label{definition2.4}
    Let $G$ be a directed/undirected (multi-)graph. Let $\H$ be a collection of finite subgraphs of $G$. The \textit{primary vertex-deletion topological structure} on $\H$ is the super-hypergraph structure defined as above.
\end{definition}

Similarly to clique complexes on multi-graphs, $\Delta(\H)$ may not be a simplicial complex in general. Therefore, the notion of a super-hypergraph is the most natural and suitable topological description for $\H$. 

The super-hypergraph $(\H, \Delta(\H))$ has a structure of fibrewise topology as follows.

Let
$$
V(\H)=\{V(H) \ | \ H\in \H\} \textrm{ and } V(\Delta(\H))=\{V(H) \ | \ H\in \Delta(\H)\}
$$
be a family of finite subsets of $V(G)$. Then $V(\Delta(\H))$ is a simplicial complex, and $V(\H)$ is a hypergraph whose simplicial closure is $V(\Delta(\H))$. Moreover we have a $\Delta$-map
$$
V\colon \Delta(\H)\longrightarrow V(\Delta(\H))
$$
and a morphism of super-hypergraphs
$$
V\colon \H\longrightarrow V(\H).
$$
By taking geometric realization, we have a continuous map
$$
|V|\colon |\Delta(\H)|\longrightarrow |V(\Delta(\H))|
$$
which is a fibrewise topology in the sense of James~\cite{James}.

Clique complexes are typical examples of primary vertex-deletion topology, where $\H$ is given by cliques in a grpah $G$. In this case, $\H$ itself is already a $\Delta$-set so $\H=\Delta(\H)$ and the map $V\colon \H\to V(\H)$ is an isomorphism.

The neighborhood complex is another good example that admits a fibrewise topological structure as follows. Let
\begin{equation}\label{equation2.3}
\widetilde{\mathcal{N}(G)}=\{ H \ | \ H \textrm{ is a subgraph of } G \textrm{ and } V(H)\in \mathcal{N}(G)\}.
\end{equation}
Then it is straightforward to check that
\begin{equation}\label{equation2.4}
\widetilde{\mathcal{N}(G)}=\Delta(\widetilde{\mathcal{N}(G)})\textrm{ and }V(\widetilde{\mathcal{N}(G)})=\mathcal{N}(G)
\end{equation}
with a continuous map
\begin{equation}\label{equation2.5}
|V|\colon |\widetilde{\mathcal{N}(G)}|\longrightarrow |\mathcal{N}(G)|
\end{equation}
which is called a \textit{fibrewise neighborhood topology} of $G$.

\subsubsection{Secondary vertex-deletion topology}
Consider the path complex of a simple digraph $G$ and its face operation $d_i$. Let $\lambda=v_0\alpha_1v_1\alpha_2v_2\cdots\alpha_nv_n$ be a directed path. Then $d_i(\lambda)$ is given by deleting the vertex $v_i$. However, we have to add back the directed edge from $v_{i-1}$ to $v_{i+1}$ provided that it exists to ensure that $d_i(\lambda) \in \mathcal{P}_n$. This gives a different type of topological structure, in which we need to redefine the edges to match the vertex removal of the face operation. This can be generalized in the following way.

Let $G$ be a directed/undirected simple graph and let $\H$ be a family of finite subgraphs of $G$. Let the vertex set $V(G)$ be totally ordered. For $H\in \H$, as a finite subgraph of $G$ with vertices $v_0<v_1<\cdots<v_n$, define $d_iH$ to be the subgraph of $G$ by removing $v_i$ from $H$ and adding the edge between $v_{i-1}$ and $v_{i+1}$ if it exists. Then $\H$ forms a super-hypergraph in a similar way as in the case of primary vertex-deletion topology. Here the notion of a super-hypergraph is necessary because there could be two subgraphs in $\H$ sharing same vertices. For instance, if there is an edge joining two distinct vertices $v$ and $w$ in $G$, then the subgraphs consist of two vertices $v$ and $w$ with the edge joining them and without the edge joining them, respectively, are different.

\begin{definition}\label{definition2.5}
Let $G$ be a directed/undirected simple graph. Let $\H$ be a collection of finite subgraphs of $G$. The \textit{secondary vertex-deletion topological structure} on $\H$ is the super-hypergraph structure defined as above.
\end{definition}

The secondary vertex-deletion topology naturally applies to subgraphs of a simple graph. However,  to construct a topological structure on a space of subgraphs of a multi-graph in this a way would be more complicated.

There are other possible topological structures on special families of subgraphs. Analogously to various techniques developed in simplicial homotopy theory,  for special families of subgraphs having good patterns, one could delete more than one vertex under each elementary face operation $d_i$.

\subsection{Edge-deletion topology}

Let $G$ be a directed/undirected (multi-)graph and $\H$ be a collection of finite subgraphs of $G$. Another reasonable way to assign the grading function $f_e\colon \H\longrightarrow \mathbb{N}=\{0,1,2,\ldots\}$ is by counting edges, that is, for $H\in\H$ let $f_e(H)=|E(H)|-1$. Then $\H_n=f^{-1}_e(n)$. Note that a subgraph $H$ of $G$ is uniquely determined by its edge set $E(H)$. We do not need to use the notion of a $\Delta$-set for describing topological structure on $\H$ from edge-deletion. If $\H$ is closed under edge-deletion operation, then it forms a simplicial complex, which is exactly a path complex in the sense of Jonsson. Otherwise, $\H$ is only a hypergraph.

For a fixed graph $G$, the edge-deletion topology could be quite different from the vertex-deletion topology because already the grading functions $f_v$ and $f_e$ could be quite different. The edge-deletion operation may not commute with the vertex-deletion operation, so the relationship between the edge-deletion topology and the vertex-deletion topology is not immediately clear. To better understand these structures, more exploration of the relationship between different topological structures on families of subgraphs is needed.

Finally, we should point out that there are many other ways to introduce topological structures on subgraphs, for example following ideas related to Hom complexes. The frontier of research in topological combinatorics has potential to provide new mathematical tools in data science.

\section{Homology Theory on Super-hypergraphs}\label{homology}
Recently, homology of hypergaphs has opened new avenues for using topological tools in data analysis. Hypergraphs have been used for data analytics in various areas of sciences from social networks to molecular bioscience. In Section~\ref{sec:topstructures} we defined super-hypergraphs. These objects are important for understanding the different explorations of topological structures on spaces of subgraphs. This setting realizes our aim to establish a unified approach to explore data science using topological combinatorics. The purpose of this section is to establish a homology theory of super-hypergraphs as a natural extension of simplicial homology and homology of hypergraphs.

\subsection{Algebraic Lemmas.}
The following algebraic tools will be needed to define a homology theory of super-hypergraphs. Although we will only make use of chain complexes of abelian groups, we note that using simplicial group models, homotopy groups can be combinatorially defined using Moore chain complexes, which are chain complexes of possibly non-abelian groups~\cite{Curtis,Wu}. There are many studies of the homotopy type of topological structures of subgraphs as indicated in the references in Section~\ref{sec:topstructures}. Therefore, we consider chain complexes of possibly non-abelian groups so that the results in this subsection may be relevant for future research.

A graded group $G_*=\{G_n\}_{n\in\Z}$ is a sequence of groups $G_n$ indexed by the integers. A graded subgroup $G'_*=\{G'_n\}_{n\in \Z}$ of $G_*=\{G_n\}_{n\in\Z}$ is a sequence of subgroups $G'_n$ such that $G'_n \leq G_n$ for $n\in\Z$. A chain complex $G_*$ of groups is a graded group $G_*$ with a group homomorphism $\partial_n=\partial_n^{G_*}\colon G_n\lra G_{n-1}$ for $n\in\Z$ such that the composite
$$
\partial_{n-1}\circ\partial_n\colon G_n\longrightarrow G_{n-2}
$$
is the trivial homomorphism. Let us emphasise that in this definition we do not require $G_n$ to be abelian. A subcomplex $C_*$ of $G_*$ is a graded subgroup $C_*$ of $G_*$ such that
$$
\partial^{G_*}_n(C_n)\subseteq C_{n-1}
$$
for each $n\in\Z$. So $C_*$ together with the restrictions $\partial^{G_*}_n|_{C_n}\colon C_n\lra C_{n-1}$ forms a chain complex.

\begin{definition}
Let $G_*$ be a chain complex of groups and let $D_*$ be a graded subgroup of $G_*$. Define
\begin{align*}
    \sup^{G_*}_*(D_*)&=\bigcap\{C_* \ | \ D_n \leq C_n \textrm{ for } n\in\Z, \textrm{ and } C_* \textrm{ is a subcomplex of } G_*\}\\
    \inf^{G_*}_*(D_*)&=\prod\{E_* \ | \ E_n \leq D_n \textrm{ for } n\in \Z, \textrm{ and } E_* \textrm{ is a subcomplex of } G_*\}.
\end{align*}
For simplicity, if the embedding of $D_*\subseteq G_*$ is clear, we denote $\sup^{G_*}_*(D_*)$ by $\sup_*(D_*)$ and $\inf^{G_*}_*(D_*)$ by $\inf_*(D_*)$.
\end{definition}

\begin{proposition}\label{proposition9.2}
Let $G_*$ be a chain complex of groups and let $D_*$ be a graded subgroup of $G_*$. Then
\begin{enumerate}
\item[{\normalfont (1)}] $\sup_*(D_*)$ is the smallest subcomplex of $G_*$ containing $D_*$. Moreover, $$\sup_n(D_*)=D_n\cdot \partial^{G_*}_{n+1}(D_{n+1})$$ is the product of $D_n$ and $\partial^{G_*}_{n+1}(D_{n+1})$.
\item[{\normalfont (2)}] $\inf_*(D_*)$ is the largest subcomplex of $G_*$ contained in $D_*$. Moreover, $$\inf_n(D_*)=D_n\cap \partial^{-1}_n(D_{n-1})$$
is the intersection of $D_n$ and $\partial^{-1}_n(D_{n-1})$.
\end{enumerate}
\end{proposition}
\begin{proof}

(1) The first part follows from the definition. Let
$$
\tilde D_n=D_n\cdot \partial^{G_*}_{n+1}(D_{n+1})
$$
for $n\in\Z$. Let $C_*$ be any subcomplex of $G_*$ such that $D_n\leq C_n$ for each $n\in\Z$. Then
$$
\partial^{G_*}_{n+1}(D_{n+1})\leq \partial^{G_*}_{n+1}(C_{n+1})\leq C_n
$$
and so
$$
\tilde D_n=D_n\cdot \partial^{G_*}_{n+1}(D_{n+1})\leq C_n.
$$
Thus $\tilde D_*$ is a graded subgroup of $C_*$ for any subcomplex $C_*$ of $G_*$ with $D_n\leq C_n$ for $n\in\Z$, and so $\tilde D_*$ is a graded subgroup of $\sup_*(D_*)$. Notice that
$$
\begin{array}{rcl}
\partial^{G_*}_n(\tilde D_*)&=&\partial^{G_*}_n(D_n\cdot \partial^{G_*}_{n+1}(D_{n+1}))\\
&\leq& \partial^{G_*}_n(D_n)\cdot \partial^{G_*}_n(\partial^{G_*}_{n+1}(D_{n+1}))\\
&=&\partial^{G_*}_n(D_n)\\
&\leq& \tilde D_{n-1}.\\
\end{array}
$$
Hence $\tilde D_*$ is a subcomplex of $G_*$ containing $D_*$, and so $\sup_*(D_*)=\tilde D_*$.

\vspace{.5cm}
(2)  The first part follows from the definition. Let
$$
\check D_n=D_n\cap \partial^{-1}_n(D_{n-1}).
$$
Let $x\in \check D_n$. Then $\partial_n(x)\in D_{n-1}$ because $x\in \partial^{-1}_n(D_{n-1})$, and
$$
\partial_n(x)\in \partial^{-1}_{n-1}(D_{n-2})
$$
because $\partial_{n-1}(\partial_n(x))=1\in D_{n-2}$. Thus $\partial_n(x)\in \check D_{n-1}$. It follows that $\check D_*$ is a subcomplex of $G_*$ contained in $D_*$. Hence
$$
\check D_*\leq \inf_*(D_*).
$$
Let $E_*$ be any subcomplex of $G_*$ such that $E_n\leq D_n$ for $n\in\Z$. Then
$$
E_n\leq \partial^{-1}_n(E_{n-1})\leq \partial^{-1}_n(D_{n-1}).
$$
Thus $E_n\leq \check D_n$ for $n\in\Z$. It follows that $\inf_*(D_*)\leq \check D_*$. This finishes the proof.
\end{proof}

Let $G_*$ be a chain complex of groups. The homology of $G_*$ is defined 
as the right cosets
$$
H_n(G_*)=\Ker(\partial^{G_*}_n)/\partial^{G_*}_{n+1}(G_{n+1}).
$$

\begin{proposition}\label{thm:inf->sup_inclus}
Let $G_*$ be a chain complex of groups and let $D_*$ be a graded subgroup of $G_*$.
\begin{enumerate}
\item[{\normalfont (1)}] The inclusion
$$
\inf_*(D_*)\longrightarrow \sup_*(D_*)
$$
induces an injective map on homology.
\item[{\normalfont (2)}] Suppose that $\partial^{G_*}_{n+1}(D_{n+1})$ is contained in the normalizer of $D_n$ in $G_n$ for each $n$. Then the inclusion
$$
\inf_*(D_*)\longrightarrow \sup_*(D_*)
$$
induces an isomorphism on homology. In particular, if $D_n$ is normal in $G_n$ for $n\in\Z$, then the inclusion $\inf_*(D_*)\longrightarrow \sup_*(D_*)$ induces an isomorphism on homology.
\end{enumerate}
\end{proposition}
\begin{proof}

(1) From Proposition~\ref{proposition9.2} (2),
$$
H_n(\inf_*(D_*))=(D_n\cap \partial^{-1}_n(D_{n-1})\cap \Ker(\partial^{G_*}_n))/\partial_{n+1}(D_{n+1}\cap \partial^{-1}_{n+1}(D_n))
$$
as right cosets. Since $\Ker(\partial^{G_*}_n)\leq \partial^{-1}_n(D_{n-1})$, we have
$$
D_n\cap \partial^{-1}_n(D_{n-1})\cap \Ker(\partial^{G_*}_n)=D_n\cap \Ker(\partial^{G_*}_n).
$$
We also claim that
$$
\partial_{n+1}(D_{n+1}\cap \partial^{-1}_{n+1}(D_n))=D_n\cap \partial_{n+1}(D_{n+1}).
$$
Clearly, $\partial_{n+1}(D_{n+1}\cap \partial^{-1}_{n+1}(D_n))\leq D_n\cap \partial_{n+1}(D_{n+1})$.

Let $x\in D_n\cap \partial_{n+1}(D_{n+1})$ and let $y\in D_{n+1}$ such that $\partial_{n+1}(y)=x$. Then $y\in D_{n+1}\cap \partial_{n+1}^{-1}(D_n)$. Thus $x\in \partial_{n+1}(D_{n+1}\cap \partial^{-1}_{n+1}(D_n))$. Hence $\partial_{n+1}(D_{n+1}\cap \partial^{-1}_{n+1}(D_n))=D_n\cap \partial_{n+1}(D_{n+1})$ and so
$$
H_n(\inf_*(D_*))=(D_n\cap \Ker(\partial^{G_*}_n))/(D_n\cap \partial_{n+1}(D_{n+1})).
$$

From Proposition~\ref{proposition9.2} (1),
$$
H_n(\sup_*(D_*))=((D_n\cdot \partial^{G_*}_{n+1}(D_{n+1}))\cap \Ker(\partial^{G_*}_n))/\partial^{G_*}_{n+1}(D_{n+1}\cdot \partial^{G_*}_{n+2}(D_{n+2})).
$$
Since $\partial^{G_*}_{n+1}(\partial^{G_*}_{n+2}(D_{n+2}))=\{1\}$,
$
\partial^{G_*}_{n+1}(D_{n+1}\cdot \partial^{G_*}_{n+2}(D_{n+2}))=\partial^{G_*}_{n+1}(D_{n+1}).
$
Thus
$$
H_n(\sup_*(D_*))=((D_n\cdot \partial^{G_*}_{n+1}(D_{n+1}))\cap \Ker(\partial^{G_*}_n))/\partial^{G_*}_{n+1}(D_{n+1}).
$$

Let $w_1,w_2\in D_n\cap \Ker(\partial^{G_*}_n)$ such that $w_1\equiv w_2$ in $H_n(\sup_*(D_*))$. Then there exists $y=\partial^{G_*}_{n+1}(D_{n+1})$ such that $w_2=w_1y$. Note that
$$
y=w_1^{-1}w_2\in D_n\cap \Ker(\partial^{G_*}_n)\leq D_n.
$$
We have $y\in D_n\cap \partial^{G_*}_{n+1}(D_{n+1})$ with $w_2=w_1y$. Thus $w_1\equiv w_2$ in $H_n(\inf_*(D_*))$. So
$$
H_n(\inf_*(D_*))\longrightarrow H_n(\sup_*(D_*))
$$
is injective. This proves (1).

(2) Let $w\in (D_n\cdot \partial^{G_*}_{n+1}(D_{n+1}))\cap \Ker(\partial^{G_*}_n)$. Then $w\in D_n\cdot \partial^{G_*}_{n+1}(D_{n+1})$ and so
$$
w=x_1y_1x_2y_2\cdots x_my_m
$$
with $x_i\in D_n$ and $y_i\in \partial^{G_*}_{n+1}(D_{n+1})$ for $1\leq i\leq m$.
Since $\partial^{G_*}_{n+1}(D_{n+1})$ is contained in the normalizer of $D_n$, the product
$$
w=x_1(y_1x_2y_1^{-1})(y_1y_2x_3y_2^{-1}y_1^{-1})\cdots (y_1\cdots y_{m-1}x_m y_{m-1}^{-1}\cdots y_1^{-1}) y_1\cdots y_m=xy
$$
with
$$
x=x_1(y_1x_2y_1^{-1})(y_1y_2x_3y_2^{-1}y_1^{-1})\cdots (y_1\cdots y_{m-1}x_m y_{m-1}^{-1}\cdots y_1^{-1})\in D_n
$$
and
$$
y=y_1y_2\cdots y_m\in \partial^{G_*}_{n+1}(D_{n+1}).
$$
Since $y\in \Ker(\partial^{G_*}_n)$ and $w\in \Ker(\partial^{G_*}_n)$,
$$
x=wy^{-1}\in \Ker(\partial^{G_*}_n).
$$
It follows that $x\in D_n\cap \Ker(\partial^{G_*}_n)$, and so
$$
H_n(\inf_*(D_*))\longrightarrow H_n(\sup_*(D_*))
$$
is surjective. From (1), $H_n(\inf_*(D_*))\longrightarrow H_n(\sup_*(D_*))$ is injective and so it is an isomorphism. This finishes the proof.
\end{proof}

\subsection{Hypergraphs.} Let $\H=(V_\H,\mathcal{E}_\H)$ be a hypergraph and let $\mathcal{P}(V_\H)$ be the set of all finite subsets of $V_\H$. The hypothesis in the definition of hypergraph $\H=(V_\H,\mathcal{E}_\H)$ only requires that $\mathcal{E}_\H\subseteq \mathcal{P}(V_\H)\smallsetminus \emptyset$. This is different from the notion of an abstract simplicial complex as hypergraphs do not require $\mathcal{E}_\H$ to be closed under taking subsets.

\begin{definition}\label{simplicial closure}
The \textit{simplicial closure} (or the \textit{associated simplicial complex} as in~\cite{Parks-Lipscomb}) of a hypergraph $\H=(V_\H,\mathcal{E}_\H)$, denoted by $\Delta\H$, is defined as
$$
\Delta\H=\{A\not=\emptyset \ | \ A\subseteq B \text{ for some } B\in \mathcal{E}_\H\}.
$$
\end{definition}

It is straightforward to check that the simplicial closure of $\H$ is the minimal simplicial complex containing $\H$. The homology of $\Delta \H$ has been studied previously in~\cite{Parks-Lipscomb}. However, it is desirable for a homology theory of $\H$ to be directly derived from $\H$ itself rather than the simplicial closure $\Delta\H$. Using Proposition~\ref{thm:inf->sup_inclus}, there is an embedded homology theory of hypergraphs that is an extension of simplicial homology theory.

\begin{definition}\label{def:emb_homology}
Let $\H=(V_\H,\mathcal{E}_\H)$ be a hypergraph with a total ordering on $V_\H$ and let $G$ be an abelian group. Let $C_*(\Delta \H;G)$ be the chain complex with coeffcients in group $G$.  Consider $\Z(\H)\otimes G$ as a graded subgroup of the chain complex of abelian groups $C_*(\Delta \H;G)$. The \textit{embedded homology $H^{\emb}_*(\H;G)$ with coefficients in $G$} is defined by
$$
H^{\emb}_*(\H;G)=H_*(\inf^{C_*(\Delta\H;G)}_*(\Z(\H)\otimes G))\cong H_*(\sup^{C_*(\Delta\H;G)}_*(\Z(\H)\otimes G)).
$$
\end{definition}

The crucial point is that by Proposition~\ref{thm:inf->sup_inclus} the inclusion
$$\inf^{C_*(\Delta\H;G)}_*(\Z(\H)\otimes G)\longrightarrow \sup^{C_*(\Delta\H;G)}_*(\Z(\H)\otimes G)$$
induces an isomorphism on homology. Hence this homology can be considered as a natural topological invariant of $\H$. To detect more subtle information about $\H$, one could explore the acyclic chain complex
$$
\sup^{C_*(\Delta\H;G)}_*(\Z(\H)\otimes G)/\inf^{C_*(\Delta\H;G)}_*(\Z(\H)\otimes G).
$$
For example, when $G$ is a field, one can investigate the Hilbert-Poincar\'e series
$$
\xi^{\emb}(\H,t)=\sum_{n=0}^\infty (\dim (\sup^{C_*(\Delta\H;G)}_n(\Z(\H)\otimes G)/\inf^{C_*(\Delta\H;G)}_n(\Z(\H)\otimes G)))t^n
$$
to detect gaps and get more robust information.

Let $\delta(\H)$ denote the maximal simplicial complex contained in $\H$. In general, $H^{\emb}_*(\H;G)$ is different from $H_*(\delta(\H);G)$ and $H_*(\Delta(\H);G)$ as shown in the following example.
\begin{example}
Let $\H$ be the boundary of a $2$-simplex with all vertices removed, $V_\H=\{0,1,2\}$ and $\mathcal{E}_\H= \{\{\{0,1\}, \{0,2\}, \{1,2\}\}\}$ as depicted in Figure~\ref{fig:hypergraph_homology_H}. Then $\delta(\H)$ is the empty set, and $\Delta(\H)$ is the boundary of the $2$-simplex. By definition, $H^{\emb}_1(\H;\Z)=\Z$ and $H^{\emb}_0(\H;\Z)=0$. Thus $H^{\emb}_*(\H;\Z)$ is different from $H_*(\delta(\H);\Z)$ and $H_*(\Delta(\H);\Z)$.
	\begin{figure}[h]
			\centering
			\begin{minipage}{0.4\textwidth}
			\begin{tikzpicture} [scale=1, inner sep=2mm]
			\coordinate (1) at (0,0);
			\coordinate (2) at (1,2);
			\coordinate (3) at (2,0);
			\node[circle,draw=black, dashed, fill=white, inner sep=0pt,minimum size=5pt, label=below left: 0] at (0,0) {};
				 \node[circle,draw=black, dashed, fill=white, inner sep=0pt,minimum size=5pt, label=above: 2] at (1,2) {};
				 \node[circle,draw,dashed, fill=white, inner sep=0pt,minimum size=5pt, label=below right: 1] at (2,0) {};
			\draw (1) -- (2) --(3) -- (1);
			\end{tikzpicture} 
			\label{fig:hypergraph_homology_H}
			\subcaption{The hypergraph $\H$, where the cross indicates that a vertex is missing.}
			\end{minipage}
			\begin{minipage}{0.4\textwidth}
				\centering
				\begin{tikzpicture} [scale=1, inner sep=2mm]
				\coordinate (1) at (0,0);
				\coordinate (2) at (2,0);
				\coordinate (3) at (1,2);
				\draw (1) node[below left] {$0$}-- (2) node[below right] {$1$} -- (3)  node[above] {$2$}-- (1);
				\foreach \i in {1,...,3} {\fill (\i) circle (1.8pt);}; 
				\end{tikzpicture}  
				\subcaption{$\Delta(\H)$, the smallest $\Delta$-set that contains $\H$ .}
				\label{fig:exam_3_Nbhd(g)}
			\end{minipage}
			\caption{}
		\end{figure}
\end{example}
This example shows that $H^{\emb}_*(\H;G)$ may not be the homology of any simplicial complex as $H^{\emb}_0(\H;\Z)=0$, which is not the case for any nonempty simplicial complex. Let us consider another example.

\begin{example}
Let $\H=(V_\H,\mathcal{E}_\H)$ with $V_\H=\{0,1,2\}$ ordered by $0<1<2$, and $$\mathcal{E}_\H=\{\{0,1,2\},\{0,1\},\{0,2\},\{0\},\{1\},\{2\}\}$$ see Figure~\ref{fig:hypergraph_homology_H}.  Then $\Delta\H$ is the abstract simplicial complex of a $2$-simplex with vertices labeled by $0,1,2$. The $1$-face $\{1,2\}$ is not in $\H$.
 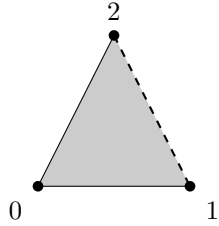
\begin{figure}[h]
			\centering
    			\begin{tikzpicture} [scale=1, inner sep=2mm]
    			\coordinate (1) at (0,0);
    			\coordinate (2) at (1,2);
    			\coordinate (3) at (2,0);
    
    			\draw (2)node[above]{$2$} -- (1) node[below left]{$0$} --(3)node[below right]{$1$};
    			
			\foreach \i in {1, 2, 3} {\fill (\i) circle (2pt);} 
    			\draw[thick, dashed](2) -- (3);
    			\fill[opacity = 0.2](1)--(2)--(3)--(1);
    			\end{tikzpicture} 
    			\caption{The hypergraph $\H$ is a standard 2 simplex where the dotted edge is missing.}
    				\label{fig:hypergraph_homology_H}
		\end{figure}
Let $G=\Z$. Then the chain complex $C_*(\Delta\H)$ is given by $C_0(\Delta\H)=\Z^{\oplus 3}=\Z\{\{0\},\{1\},\{2\}\}$, $C_1(\Delta\H)=\Z^{\oplus 3}=\Z\{\{0,1\},\{0,2\},\{1,2\}\}$, and $C_2(\Delta\H)=\Z=\Z\{\{0,1,2\}\}$.

We have $\inf_0=C_0(\Delta\H)=\Z\{\{0\},\{1\},\{2\}\}$,
$$
\inf_1=\Z(\mathcal{E}_1)\cap \partial_1^{-1}(\Z(\mathcal{E}_0))=\Z(\mathcal{E}_1)\cap C_1(\Delta\H)=\Z(\mathcal{E}_1)=\Z\{\{0,1\},\{0,2\}\}
$$
$$
\inf_2=\Z(\mathcal{E}_2)\cap \partial_2^{-1}(\Z(\mathcal{E}_1))=0
$$
with $\partial_1(\inf_1)=\Z\{\{1\}-\{0\},\{2\}-\{0\}\}$. Thus $H^{\emb}_0(\H)=\Z$ and $H^{\emb}_i(\H)=0$ for $i\geq1$.

Let $\H'=(V_{\H'},\mathcal{E}_{\H'})$ with $V_{\H'}=V_\H=\{0,1,2\}$ ordered by $0<1<2$, and $$\mathcal{E}_{\H'}=\{\{0,1,2\},\{0,1\},\{0\},\{1\},\{2\}\}.$$ Then $H^{\emb}_0(\H')=\Z\oplus \Z$ and $H^{\emb}_i(\H)=0$ for $i\geq1$. Thus the embedded homology of $\H'$ can not be realized as the homology of a path-connected topological space.

\end{example}

These examples indicate that embedded homology is a new homology theory with unusual properties and that poses its own questions and challenges.

The definition of embedded homology of a hypergraph $\H$ depends on the orientation of its simplicial closure $\Delta\H$. It is well-known that simplicial homology is independent on the choice of orientation. The following theorem shows that this is also true for the embedded homology of hypergraphs.

\begin{theorem}\label{thm:hyp-grph_orienation}
The embedded homology $H^{\emb}_*(\H;G)$ of a hypergraph $\H$ does not depend on a choice of orientation on $\Delta\H$.
\end{theorem}

\begin{proof}
Let $H_*^{\emb}(\H)$ and $G(\H)$ denote $H_*^{\emb}(\H,G)$ and $\Z(\H)\otimes G$, respectively. We assume that $V_\H$ is a finite set $\{v_1,v_2,\ldots,v_m\}$. Take a linear ordering on $V_\H$ so that $v_1<v_2<\cdots<v_n$ as a fixed choice of total order and let $C_*=C_*(\Delta\H;G)$ denote the oriented chain complex. It suffices to show that the homology stays the same up to isomorphism under the transpositions $(i,i+1)$ of the ordering on $V(\H)$ for $1\leq i\leq m-1$.

Let $\partial'_n\colon C_n\to C_{n-1}$, $n\geq 1$ be the boundary homomorphism defined using the new order on $V_{\H}$, that is,  $v_1<v_2<\cdots<v_{i-1}<v_{i+1}<v_i<v_{i+2}<\cdots <v_n$. For $n\geq1$, the abelian group $C_n$ admits a direct sum decomposition
\begin{equation}\label{equation3.1}
C_n=C_n^{v_iv_{i+1}}\oplus C_n^{\widehat{v_iv_{i+1}}}
\end{equation}
where $C_n^{v_iv_{i+1}}$ is the subgroup of $C_n$ given by linear combinations with coefficients in group $G$ of the $n$-simplicies $\sigma\in \Delta\H$ whose vertex set contains both $v_i$ and $v_{i+1}$, and $C_n^{\widehat{v_iv_{i+1}}}$ is the subgroup of $C_n$ given by linear combinations with coefficients in group $G$ of the remaining $n$-simplicies in $\Delta\H$. For any chain $\alpha\in C_n$, there is a corresponding unique decomposition
\begin{equation}\label{equation3.2}
\alpha=\alpha^{v_iv_{i+1}}+\alpha^{\widehat{v_iv_{i+1}}}.
\end{equation}

Since $v_i$ and $v_{i+1}$ are neighbored vertices in the order, we have
$$
\partial'(\sigma)=\partial(\sigma)\textrm{ if } \sigma \textrm{ does not contain both } v_i \textrm{ and } v_{i+1} \textrm{ in its vertex set}.
$$
Therefore
\begin{equation}\label{equation3.3}
\partial'_n|=\partial_n|\colon C_n^{\widehat{v_iv_{i+1}}}\longrightarrow C_{n-1}.
\end{equation}
Let $\sigma=[a_1\cdots a_tv_iv_{i+1}b_1\cdots b_s]$ be an oriented simplex in $\Delta\H$ with $a_1<\cdots<a_t<v_i<v_{i+1}<b_1<\cdots<b_s$.
By the definition of $\partial(\sigma)$, we have
$$
\begin{array}{rcl}
\partial (\sigma)^{v_iv_{i+1}}&=&\sum\limits_{j=1}^t(-1)^{j-1}[a_1\cdots \hat a_j\cdots a_tv_iv_{j+1}b_1\cdots b_s]+\\
&&\sum\limits_{k=1}^s(-1)^{t+k+1}[a_1\cdots a_tv_iv_{i+1}b_1\cdots\hat b_k\cdots b_s]\\
\end{array}
$$
where $\cdots \hat x\cdots$ means that $x$ is deleted, and
$$
\partial (\sigma)^{\widehat{v_iv_{i+1}}}=(-1)^t[a_1\cdots a_tv_{i+1}b_1\cdots b_s]+(-1)^{t+1}[a_1\cdots a_tv_ib_1\cdots b_s].
$$
By switching the order of $v_i$ and $v_{i+1}$, we have
$$
\partial'(\sigma)=(\partial (\sigma))^{v_iv_{i+1}}-(\partial (\sigma))^{\widehat{v_iv_{i+1}}}.
$$
Extending this formula linearly with coefficients in group $G$, we obtain the formula
\begin{equation}\label{equation3.4}
\partial'(\alpha)=(\partial(\alpha))^{v_iv_{i+1}}-(\partial(\alpha))^{\widehat{v_iv_{i+1}}} \textrm{ for } \alpha\in C_*^{v_iv_{i+1}}.
\end{equation}

Define the group homomorphism
$$\phi_n\colon C_n=C_n^{v_iv_{i+1}}\oplus C_n^{\widehat{v_iv_{i+1}}}\longrightarrow C_n=C_n^{v_iv_{i+1}}\oplus C_n^{\widehat{v_iv_{i+1}}}$$
by setting
$$
\phi_n(z^{v_iv_{i+1}}+z^{\widehat{v_iv_{i+1}}})=z^{v_iv_{i+1}}-z^{\widehat{v_iv_{i+1}}}.
$$
Clearly, $\phi_n$ is an isomorphism. Let $z=z^{v_iv_{i+1}}+z^{\widehat{v_iv_{i+1}}}\in C_n$ be a chain. Then
$$
\begin{array}{rcl}
\partial(z)&=&\partial(z^{v_iv_{i+1}})+\partial(z^{\widehat{v_iv_{i+1}}})\\
&=&(\partial(z^{v_iv_{i+1}}))^{v_iv_{i+1}}+(\partial(z^{v_iv_{i+1}}))^{\widehat{v_iv_{i+1}}}+\partial(z^{\widehat{v_iv_{i+1}}})\\
\end{array}
$$
so
$$
(\partial(z))^{v_iv_{i+1}}=(\partial(z^{v_iv_{i+1}}))^{v_iv_{i+1}}
$$
and
$$
(\partial(z))^{\widehat{v_iv_{i+1}}}=(\partial(z^{v_iv_{i+1}}))^{\widehat{v_iv_{i+1}}}+\partial(z^{\widehat{v_iv_{i+1}}}).
$$
On the other hand, by direct computation
$$
\begin{array}{rcl}
\partial'(\phi_n(z))&=&\partial'(z^{v_iv_{i+1}}-z^{\widehat{v_iv_{i+1}}})\\
&=&\partial(z^{v_iv_{i+1}}))^{v_iv_{i+1}} -((\partial(z^{v_iv_{i+1}}))^{\widehat{v_iv_{i+1}}}+\partial(z^{\widehat{v_iv_{i+1}}}))\\
&=&(\partial(z))^{v_iv_{i+1}}-(\partial(z))^{\widehat{v_iv_{i+1}}}.\\
\end{array}
$$
This gives a commutative diagram
\[
\xymatrix{
C_n \ar[r]^{\phi_n}_{\cong} \ar[d]^{\partial_n}  & C_n\ar[d]^{\partial'_n} \\
C_{n-1}\ar[r]^{\phi_{n-1}}_{\cong}& C_{n-1}.\\}
\]
Note that the decomposition~(\ref{equation3.1}) restricted to $G(\H_n)$ gives the decomposition
$$G(\H)=G(\H)^{v_iv_{i+1}}\oplus G(\H)^{\widehat{v_iv_{i+1}}}$$ with the same rule on simplices. The subgroup $G(H_n)$ is invariant under $\phi_n$. Moreover, there is a commutative diagram
\[
\xymatrix{
\inf^{C_*(\Delta\H;G)}_n(G(\H)) \ar[r]^{\phi_n}_{\cong} \ar[d]^{\partial_n|}  & \inf^{C_*(\Delta\H;G)}_n(G(\H))\ar[d]^{\partial'_n|} \\
\inf^{C_*(\Delta\H;G)}_{n-1}(G(\H))\ar[r]^{\phi_{n-1}}_{\cong}& \inf^{C_*(\Delta\H;G)}_{n-1}(G(\H)).\\}
\]
The assertion then follows by taking homology of this commutative diagram.

\end{proof}

The embedded homology of hypergraphs was introduced in 2019 in~\cite{BLRW}. Previously, cohomological aspects on $k$-uniform hypergraphs have been studied, see ~\cite{C77, C78, CW86, Chung, MS75, ML83, S76, ST81, We84, Z81}, using cohomology introduced in a combinatorial way. Also Emtander~\cite{Emtander} studied the homology of the \textit{independence complex} $\Delta^c\H$ of a hypergraph $\H=(V_{\H},\mathcal{E}_\H)$ in 2009, where  $\Delta^c\H=\{F\subseteq V_\H \ | \ E\not\subseteq F \textrm{ for any } E\in \mathcal{E}_\H\}$. The approach of embedded homology is different from the classical research on topological structures related to hypergraphs as it is directly define on the hypergraph.

Although the embedded homology of hypergraphs is a new topic with surprising properties, it inherits many characteristics of simplicial homology. The following proposition is an example of this, the proof of this proposition is similar to that in of~\cite[Theorem 8.2, p.45]{Munkres}.

Recall that the \textit{cone} $CK$ of a simplicial complex $K$ is defined as a join $CK=w\ast K$ with $w$ a vertex not in $K$. Analogously, we can define the join of hypergraphs and the cone $C\H=w\ast\H$.

\begin{theorem}\label{hypergraph cone theorem}
Let $\H$ be a hypergraph and let $G$ be an abelian group. Then
$$
H_n^{\emb}(C\H;G)=\left\{
\begin{array}{lcl}
0&\textrm{ if }& n>0\\
G&\textrm{ if }&n=0.\\
\end{array}\right.
$$
\hfill $\Box$
\end{theorem}

\subsection{Super-Hypergraphs}\label{super-hypergraph}
\subsubsection{Homology of Super-Hypergraphs}
Recall from Definition~\ref{def:sup_hyp_graph} that a super-hypergraph is a pair $(\H, X)$, where $X$ is a $\Delta$-set and $\H$ is a graded subset of $X$. Here $X$ is a parental $\Delta$-set of $\H$ so that the hypergraph $\H$ is born from $X$. Using Proposition~\ref{thm:inf->sup_inclus}, there is an embedded homology on super-hypergraphs.

\begin{definition}\label{super-hypergraph homology}
Let $(\H,X)$ be a super-hypergrapp and let $G$ be an abelian group. The \textit{embedded homology $H_*^{\emb,X}(\H;G)$ with coefficients in $G$} of $(\H,X)$ is defined by
$$
H_*^{\emb,X}(\H;G)=H_*(\inf^{C_*(X;G)}_*(\Z(\H)\otimes G))\cong H_*(\sup^{C_*(X;G)}_*(\Z(\H)\otimes G))
$$
\end{definition}
where $\Z(\H)\otimes G$ is a graded subgroup of the chain complex of abelian groups $C_*(X; G)$.

We want to show that this definition is an extension of the embedded homology of hypergraphs. An \textit{oriented hypergraph} is a hypergraph $\H$ with a partial order on its vertex set so that the restriction to the vertices of each hyperedge of $\H$ is linear. If the vertices of a simplex are totally ordered, then the restricted order on the vertices of any of its faces is linear. Thus the simplicial closure $\Delta\H$ can be oriented with its orientation induced by the order on $\H$. From Definition~\ref{def:emb_homology},
$$
H_*^{\emb}(\H;G)=H_*^{\emb,\Delta\H}(\H;G)
$$
by considering the oriented simplicial complex $\Delta\H$ as a $\Delta$-set. From Theorem~\ref{thm:hyp-grph_orienation}, this definition is independent on the choice of orientation.

We now consider how morphisms of super-hypergraphs, recall Definition~\ref{def:sup_hyp_graph_morphisms}, induce maps on the infimum and supremum chain complexes as well as embedded homology of super-hypergraphs.

\begin{proposition}\label{proposition3.12}
Let $\phi\colon (\H, X)\lra (\H', Y)$ be a morphism of super-hypergraphs. Then there is a commutative diagram
\begin{equation}\label{CD:morphism_SHG1}
\xymatrix{
\inf^{C_*(X;G)}_*(\Z(\H)\otimes G) \ar@{}[r]|-*[@]{\subseteq} \ar[d]^{\phi_{\#}|}  & \Z(\H)\otimes G\ar@{}[r]|-*[@]{\subseteq}\ar[d]^{\phi_{\#}|} & \sup^{C_*(X;G)}_*((\Z(\H)\otimes G)\ar@{^{(}->}[r]\ar[d]^{\phi_{\#}|}& C_*(X;G)\ar[d]^{\phi_{\#}}\\
\inf^{C_*(Y;G)}_*(\Z(\H')\otimes G)\ar@{}[r]|-*[@]{\subseteq} & \Z(\H')\otimes G\ar@{}[r]|-*[@]{\subseteq} &\sup^{C_*(Y;G)}_*((\Z(\H')\otimes G)\ar@{^{(}->}[r]& C_*(Y;G)}
\end{equation}
which induces a map $\phi_*\colon  H_*^{\emb,X}(\H;G)\lra  H_*^{\emb,Y}(\H';G)$. Moreover, if $\phi(\H)=\H'$, then there is a short exact sequence of chain complexes
\begin{equation}\label{CD:morphism_SHG2}
\xymatrix{
\sup^{C_*(X;G)}_*((\Z(\H)\otimes G)\cap \Ker (\phi_{\#})\ar@{^{(}->}[r] &\sup^{C_*(X;G)}_*((\Z(\H)\otimes G)\ar@{->>}[r]^{\phi_{\#}} &\sup^{C_*(Y;G)}_*((\Z(\H')\otimes G)}.
\end{equation}
\end{proposition}

\begin{proof}
Since $\sup^{C_*(Y;G)}_*((\Z(\H')\otimes G)$ is a subcomplex of $C_*(Y;G)$ containing $\Z(\H')\otimes G$, its preimage
\[
\phi^{-1}_{\#}(\sup^{C_*(Y;G)}_*((\Z(\H')\otimes G))
\]
is a subcomplex of $C_*(X;G)$ containing $\Z(\H)\otimes G$ because $\phi(\H)\subseteq \H'$. Thus
\begin{equation}\label{proof-proposition3.12}
\sup^{C_*(X;G)}_*((\Z(\H)\otimes G)\subseteq \phi^{-1}_{\#}(\sup^{C_*(Y;G)}_*((\Z(\H')\otimes G))
\end{equation}
as a subcomplex. Similarly
\begin{equation}\label{proof-proposition3.12-2}
\phi_{\#}(\inf^{C_*(X;G)}_*(\Z(\H)\otimes G))\subseteq \inf^{C_*(Y;G)}_*(\Z(\H')\otimes G)
\end{equation}
as a subcomplex.
Therefore, there is a commutative diagram \ref{CD:morphism_SHG1}.

Now, we assume that $\phi(\H)=\H'$. Then
$$
\phi_{\#}(\sup^{C_*(X;G)}_*((\Z(\H)\otimes G))\supseteq \phi_{\#}((\Z(\H)\otimes G))=\Z(\H')\otimes G
$$
is a subcomplex of $C_*(Y;G)$ containing $\Z(\H')\otimes G$. Hence
$$
\phi_{\#}(\sup^{C_*(X;G)}_*((\Z(\H)\otimes G))\supseteq
\sup^{C_*(Y;G)}_*((\Z(\H')\otimes G).
$$
Together with the containment ~\ref{proof-proposition3.12}, we have
$$
\sup^{C_*(Y;G)}_*((\Z(\H')\otimes G)= \phi_{\#}(\sup^{C_*(X;G)}_*((\Z(\H)\otimes G))
$$
and hence the short exact sequence \ref{CD:morphism_SHG2}.
\end{proof}

\begin{corollary}\label{Corollary3.13}
Let $\phi\colon (\H, X)\lra (\H', Y)$ be a morphism of super-hypergraphs. Suppose that
\begin{enumerate}
\item[{\normalfont (1)}] $\phi\colon X\lra Y$ is an injective $\Delta$-map and
\item[{\normalfont (2)}] $\phi(\H)=\H'$.
\end{enumerate}
Then
$$
\phi_{\#}\colon \sup^{C_*(X;G)}_*((\Z(\H)\otimes G)\longrightarrow \sup^{C_*(Y;G)}_*((\Z(\H')\otimes G)
$$
is an isomorphism. In particular, $\phi_*\colon  H_*^{\emb,X}(\H;G)\lra  H_*^{\emb,Y}(\H';G)$ is an isomorphism.
\end{corollary}
\begin{proof}
By the assumption (1),
$$
\sup^{C_*(X;G)}_*((\Z(\H)\otimes G)\cap \Ker (\phi_{\#})=0
$$
and so the assertion follows by Proposition~\ref{proposition3.12}.
\end{proof}

\subsubsection{Variations of Parental $\Delta$-sets}
In recent topological applications to data analytics and machine learning, one of the most common approaches is to use discrete Hodge-Laplacian theory. Mathematically, combinatorial Laplacian operators defined on linear transformations on cochains of simplicial complexes have been studied, for example in ~\cite{DuvalandReiner, HorakandJost}. Therefore, in addition to simplicial homology, research on (co)chains of simplicial complexes such as spectral analysis on combinatorial Laplacian operators is also important for potential applications in data science. Similarly, the research on chains $\inf^{C_*(X;G)}_*(\Z(\H)\otimes G)$ and $\sup^{C_*(X;G)}_*(\Z(\H)\otimes G)$ could be useful for the applications of super-hypergraphs. Next, we describe some basic properties related to the chain complexes arising from super-hypergraphs.

A super-hypergraph $\H$ is assumed to have a parental $\Delta$-set $X$ that carries geometric structural information about $\H$. The embedded homology $H_*^{\emb,X}(\H;G)$ is defined using the geometric information inherited from $X$. On level of chains, there are inclusions of graded groups
$$
\xymatrix{
\inf^{C_*(X;G)}_*(\Z(\H)\otimes G)\ar@{^{(}->}[r]& \Z(\H)\otimes G\ar@{^{(}->}[r]&\sup^{C_*(X;G)}_*(\Z(\H)\otimes G)\ar@{^{(}->}[r]&C_*(X;G)\\}
$$
where the right most inclusion is a chain map. By Proposition~\ref{thm:inf->sup_inclus}, the inclusion
$$
\xymatrix{
\inf^{C_*(X;G)}_*(\Z(\H)\otimes G)\ar@{^{(}->}[r]&\sup^{C_*(X;G)}_*(\Z(\H)\otimes G)}
$$
is a chain homotopy equivalence that defines the embedded homology. The gap complex
\begin{equation}\label{equation-gap}
\sup^{C_*(X;G)}_*(\Z(\H)\otimes G)/\inf^{C_*(X;G)}_*(\Z(\H)\otimes G)
\end{equation}
which is an acyclic chain complex, gives more robust information about the graded set $\H$.

Let $\H$ be a fixed graded data set. Our aim is to vary the parental $\Delta$-set $X$ such that the corresponding infimum and supremum chain complexes reveal different aspects of the topological structure of $\H$. One natural way to vary the parental $\Delta$-set would be to consider morphisms $\phi\colon (\H, X)\to (\H,Y)$ that fix $\H$ and investigate how these affect the embedded homology. Another important question is whether one can vary the parental $\Delta$-set so that the gap complex~(\ref{equation-gap}) is as small as possible. We consider the following example.

\begin{example}\label{examp:diff_delta=>diff_hom}
Let $n$ be an odd positive integer. Let $X=\Delta^+[n]$ be the $\Delta$-set induced by an $n$-simplex with vertices labelled $0,1,\ldots,n$. Let $Y$ be the $\Delta$-set with $Y_k=\{a_k\}$, $0\leq k\leq n$, $Y_k=\emptyset$ for $k>n$ and $d_i(a_k)=a_{k-1}$ for $0\leq i\leq k\leq n$.  Let $\H$ be the graded set given by $\H_n=\{x_n\}$ and $\H_k=\emptyset$ for $k\not=n$. Consider the super-hypergraphs $(\H,X)$, $x_n=[0,1,2,\ldots,n]$, and $(\H,Y)$, $x_n=a_n$. There is a unique morphism $\phi\colon (\H,X)\to (\H,Y)$ such that $\phi|_{\H}=\mathrm{id}_{\H}$. Then, we have  $H_k^{\emb,X}(\H;\Z)=0$ for $k\geq0$ and
$$
 H_k^{\emb,Y}(\H;\Z)=\left\{
 \begin{array}{lcl}
 \Z&\textrm{ if } k=n,\\
 0&\textrm{ otherwise.}\\
 \end{array}\right.
 $$
This shows the embedded homology or super-hypergraphs depends on the parentl $\Delta$-set. The gap complex for $(\H,X)$ is the same as the acyclic complex
$$
\sup^{C_*(X;\Z)}_k(\Z(\H))=\left\{
\begin{array}{lcl}
\Z&\textrm{ if } k=n, n-1,\\
0&\textrm{ otherwise}\\
\end{array}\right.
$$
with $\partial \colon \sup^{C_*(X;\Z)}_n(\Z(\H))\to \sup^{C_*(X;\Z)}_{n-1}(\Z(\H))$ an isomorphism, and the gap complex for $(\H,Y)$ is $0$.
\end{example}

Now we consider the effect of morphisms on super-hypergraphs with a fixed hypergraph more generally. Let $\phi\colon (\H,X)\to (\H,Y)$ be a morphism of super-hypergraphs so that $\phi|_{\H}=\mathrm{id}_{\H}$. By ~\ref{proof-proposition3.12-2}, we have
$$
\xymatrix{
\inf^{C_*(X;G)}_*(\Z(\H)\otimes G)\ar@{^{(}->}[r]&\inf^{C_*(Y;G)}_*(\Z(\H)\otimes G)}
$$
namely the chain complex of $Y$ is closer to the graded group $\Z(\H)\otimes G$ than the chain complex of $X$. Together with the inclusion ~\ref{proof-proposition3.12}, it follows that the gap complex for $(\H,Y)$ is smaller, as shown in the above example. If $\phi$ is injective, then both
$$
\inf^{C_*(X;G)}_*(\Z(\H)\otimes G)=\inf^{C_*(Y;G)}_*(\Z(\H)\otimes G) \text{ and }
$$
$$
\sup^{C_*(X;G)}_*(\Z(\H)\otimes G)=\sup^{C_*(Y;G)}_*(\Z(\H)\otimes G)
$$
which means that we can replace $X$ by any of its $\Delta$-subsets  that contain $\H$ without changing the infimum and supremum chain complexes. In particular, we have

\begin{equation}\label{equ:emb_hom_delta_closure}
    H_*^{\emb,X}(\H;G)= H_*^{\emb,\Delta^X(\H)}(\H;G)
\end{equation}
where $\Delta^X(\H)$ is the $\Delta$-closure of $\H$ in $X$, which is the minimal $\Delta$-subset of $X$ containing $\H$ defined in Definition~\ref{def:sup_hyp_graph}.

\begin{definition}\label{regular super-hypergraph}
A super-hypergraph $(\H,X)$ is called \textit{regular} if $X=\Delta^X(\H)$.
\end{definition}

It is straightforward to see that a super-hypergraph $(\H,X)$ is regular if and only if all elements in $X$ are obtained from the elements in $\H$ together with their iterated faces in $X$. The following proposition may be useful for analysing for variations of parental $\Delta$-sets.

\begin{proposition}
Let $\H=\{\H_n\}_{n\geq0}$ be a graded set such that the cardinality of the set $\coprod\limits_{n\geq0} H_n$ is finite. Then there are finitely many regular super-hypergraphs $(\H,X)$ up to isomorphisms.
\end{proposition}
\begin{proof}
Consider the collection of all possible $\Delta$-sets $X$ such that $(\H,X)$ is a regular super-hypergraph. If $(\H,X)$ is regular, then all elements in $X$ are given by the elements in the graded subset $\H$ together with their iterated faces in $X$. Therefore, as a $\Delta$-set $X$ is a finite extension of $\H$ together with finitely many face operations.
\end{proof}

\begin{definition}\label{complete super-hypergraph}
A super-hypergraph $(\H,X)$ is \textit{complete} if it is regular, and for any morphism of super-hypergraphs $\phi\colon (\H,X)\to (\H,Y)$ with $\phi|_{\H}=\mathrm{id}_{\H}$, the $\Delta$-map $\phi\colon X\to Y$ is injective.
\end{definition}

Therefore, for a complete super-hypergragh $(\H,X)$, the infimum complex $\inf^{C_*(X;G)}_*(\Z(\H)\otimes G)$ reaches a maximum and the gap complex~(\ref{equation-gap}) reaches a minimum.

By definition, a complete super-hypergraph is regular. However, the converse may not be true. For instance, the super-hypergraph $(\H,X)$ in Example~\ref{examp:diff_delta=>diff_hom} is regular but not complete as the morphism to $(\H,Y)$ is not an injective $\Delta$-set map. Therefore completeness provides a notion of maximality in the set of super-hypergraphs related to a given hypergraph.

\begin{theorem}[Completeness Criterion]
A regular super-hypergraph $(\H, X)$ with $\H\not=\emptyset$ is complete if and only if it has the following properties:
\begin{enumerate}
\item (Vertex Property) If $\H_0\not=\emptyset$, then $X_0=\H_0$. If $\H_0=\emptyset$, then $X_0$ is a one-point set.
\item (Matching Face Property) Let $z_1,z_2\in X_n$ with $n>0$. Suppose that
\begin{enumerate}
\item[(i)] $d_iz_1=d_iz_2$, $0\leq i\leq n$, and
\item[(ii)] $\{z_1,z_2\}\not\subseteq \H_n$.
\end{enumerate}
Then $z_1=z_2$.
\end{enumerate}
\end{theorem}
\begin{proof}
Suppose that $(\H,X)$ is complete. We first show that properties (1) and (2) hold.

(1) Assume that $\H_0\not=\emptyset$. Suppose that there exists $z\in \Delta^X(\H)_0\smallsetminus \H_0$. Choose an element $x\in\H_0$. Let $Z$ be the $\Delta$-quotient of $X$ by identifying $z$ and $x$. Let $q\colon X\lra Z$ be the quotient map. Then $q|_{\H}$ is injective, but $q\colon X \lra Z$ is not injective, which is a contradiction. Hence $X_0=\H_0$.

If $\H_0=\emptyset$, then the same argument shows that $X_0$ must be a one-point set.

(2) Suppose $z_1\not=z_2$. Then, similarly to the proof of (1), we can construct the $\Delta$-quotient $Z$ obtained from $X$ by identifying $z_1$ with $z_2$ in dimension $n$. Since all faces of $z_1$ and $z_2$ match, the equivalence relation $z_1\sim z_2$ does not induce non-trivial identifications in $X_m$ for $m\not=n$. Since $\{z_1,z_2\}\not\subseteq \H_n$, the equivalence relation $z_1\sim z_2$ does not effect elements in $\H$. This proves property (2).

Now let $\phi\colon (\H,X)\to (\H,Y)$  be a morphism of super-hypergraphs with $\phi|_{\H}=\mathrm{id}_{\H}$ such that $(\H,X)$ is a regular super-hypergraph satisfying properties (1) and (2).

We show that the $\Delta$-map $\phi\colon X\to Y$ is injective. By the vertex property, $\phi\colon X_0\to Y_0$ is injective. Suppose that $\phi\colon X\to Y$ is not injective. Then there exists $n>0$ such that $\phi\colon X_k\to Y_k$ is injective for $k<n$ and $\phi\colon X_n\to Y_n$ is not injective. It follows that there exists $z_1,z_2\in X_n$ with $z_1\not=z_2$ such that $\phi(z_1)=\phi(z_2)$. Then $\{z_1,z_2\}\not\subseteq \H_n$ because $\phi|_{\H}$ is injective. Since $\phi$ is a $\Delta$-set map, we have
$$
\phi(d_iz_1)=d_i\phi(z_1)=d_i\phi(z_2)=\phi(d_iz_2)
$$
for $0\leq i\leq n$ and $\phi\colon X_{n-1}\to Y_{n-1}$ is injective. Therefore $d_iz_1=d_iz_2$ for $0\leq i\leq n$. However, by the matching face property, $z_1=z_2$, which contradicts the assumption that $z_1\not=z_2$.
\end{proof}

For a given regular super-hypergraph $(\H,X)$, the above proof gives a way to construct a complete super-hypergraph $(\H,Y)$ with $Y$ as a $\Delta$-quotient of $X$. The following example shows that $(\H,X)$ may have non-isomorphic complete quotients.

\begin{example}
Let
$$
\H=\{\{0,1,2\}, \{0,1\}, \{0,2\},\{1,2\},\{1\},\{2\}\}
$$
be the graded subset of the $2$-simplex $X=\Delta^+[2]$ without the $0$ vertex. Let $Y$ be the $\Delta$-quotient of $X$ by identifying vertices $\{1\}$ and $\{0\}$, and let $Z$  be the $\Delta$-quotient of $X$ by identifying vertices $\{2\}$ and $\{0\}$. Then
 \begin{enumerate}
\item[1)]  $(\H, X)$, $(\H,Y)$ and $(\H,Z)$ are regular super-hypergraph.
\item[2)] $(\H,Y)$ and $(\H,Z)$ are complete, but $(\H, X)$ is not complete.
\item[3)] There are non-injective $\Delta$-quotients $X\twoheadrightarrow Y$ and $X\twoheadrightarrow Z$ with $(\H,Y)\not\cong (\H,Z)$.
\end{enumerate}
\end{example}

\subsubsection{Mayer-Vietoris Sequence}

The Mayer-Vietoris sequence (MV sequence) is one of the fundamental tools in topology for inductively computing homology. In general, the MV sequence fails for embedded homology of super-hypergraphs.

\begin{example}\label{example9.19}
Let $X=\{f_1,f_2,e_1,e_2, v\}$ be a $\Delta$-set with face operations given by
\begin{enumerate}
\item $d_0f_i=d_2f_i=e_1$, $d_1f_i=e_2$ where $j=1,2$, and
\item $d_ie_j=v$ for $i=0,1$ and $j=1,2$.
\end{enumerate}
Let
$$
\H=X\smallsetminus\{e_1,e_2\}
$$
be the graded subset of $X$ missing the edges $e_1,e_2$. Let $A=\{f_1, e_1,e_2,v\}$ and $B=\{f_2, e_1,e_2,v\}$ with face operations induced from $X$. Then
$$
(\H,X)=(\H\cap A, A)\cup (\H\cap B, B)
$$
with $A\cap B=\{e_1,e_2,v\}$ and $\H\cap A\cap B=\{v\}$. Then there is no  exact sequence 
\[\xymatrix{
\cdots \ar[r]& H_2^{\emb, A}(\H\cap A)\oplus H_2^{\emb,B}(\H\cap B)\ar[r]& H_2^{\emb, X}(\H)\ar[r]& H_1^{\emb, A\cap B}(\H\cap A\cap B)\ar[r]&\cdots}
\]
Since $\H_1=\emptyset$ and $\H_2=X_2$,
$$
\begin{array}{rcl}
\inf^{C_*(A)}_2(\Z(\H\cap A))&=& \Z(\H_2\cap A)\cap\Ker(\partial^A_2)=0,\\
\inf^{C_*(B)}_2(\Z(\H\cap B))&=& \Z(\H_2\cap B)\cap\Ker(\partial^B_2)=0,\\
\inf^{C_*(X)}_2(\Z(\H))&=&\Z(\H_2)\cap\Ker(\partial^X_2)=\Z\{f_1-f_2\}=\Z,\\
\inf^{C_*(X)}_1(\Z(\H))&=&\Z(\H_1)\cap (\partial^X_1)^{-1}(\Z(\H_0))=0,\\
\inf^{C_*(A\cap B)}_1(\Z(\H\cap A\cap B))&=&0.\\
\end{array}
$$
Then $H_2^{\emb, A}(\H\cap A)=H_2^{\emb,B}(\H\cap B)=H_1^{\emb, A\cap B}(\H\cap A\cap B)=0$ and $H_2^{\emb, X}(\H)=\Z$. Hence the above sequence cannot be exact.
\end{example}

An analogue of the classical Mayer-Vietoris sequence for super-hypergraphs is a multi-exact sequence derived from the following theorem.

\begin{theorem}\label{thm:mv_seq_SHG}
Let $(\H,X)$ be a super-hypergraph and let $A$ and $B$ be $\Delta$-subsets of $X$ such that $A\cup B=X$. Let $\H^A=\H\cap A$, $\H^B=\H\cap B$ and $\H^{A\cap B}=\H\cap A\cap B$. Let $G$ be an abelian group, and denote $\sup^{C_*(X;G)}_*(\H)$ and $\inf^{C_*(X;G)}_*(\H)$ by $\sup^X_*(\H)$ and $\inf^X_*(\H)$, respectively.

Then there is a commutative diagram
\[\xymatrix{
&&\sup^{X}_*(\H)\ar@{=}[d]\\
\sup^{A}_*(\H^A)\cap \sup^{B}_*(\H^B)\ar@{^{(}->}[r] & \sup^{A}_*(\H^A)\oplus \sup^{B}_*(\H^B)\ar@{->>}[r]^{j_A-j_B}& \sup^{A}_*(\H^A)+ \sup^{B}_*(\H^B)\\
\inf^{A}_*(\H^A)\cap \inf^{B}_*(\H^B))\ar@{^{(}->}[u]\ar@{^{(}->}[r]& \inf^{A}_*(\H^A)\oplus \inf^{B}_*(\H^B))\ar@{^{(}->}[u]^{\simeq}\ar@{->>}[r]^{j_A|-j_B|}& \inf^{A}_*(\H^A)+ \inf^{B}_*(\H^B))\ar@{^{(}->}[u]\\
\inf^{A\cap B}_*(\H^{A\cap B})\ar@{=}[u]&&}
\]
where the middle two rows are short exact sequences of chain complexes, the maps $j_A$ and $j_B$ are canonical inclusions and the vertical arrows are inclusions.

\end{theorem}
Example~\ref{example9.19} shows that the left and right vertical arrows in the above diagram are not chain homotopy equivalences. However, the middle vertical arrow is always a chain homotopy equivalence.
\begin{proof}
We need to show that $$\sup^{A}_*(\H^A)+ \sup^{B}_*(\H^B)=\sup^{X}_*(\H) \textrm{ and }$$ $$\inf^{A\cap B}_*(\H^{A\cap B})=\inf^{A}_*(\H^A)\cap \inf^{B}_*(\H^B).$$
First we prove that $\sup^{A}_*(\H^A)+ \sup^{B}_*(\H^B)=\sup^{X}_*(\H)$. Since $\sup^{A}_*(\H^A)+ \sup^{B}_*(\H^B)$ is a sub complex of $C_*(X;G)$ containing $\Z(\H)\otimes G=\Z(\H^A)\otimes G+\Z(\H^B)\otimes G$, we have $\sup^{X}_*(\H)\subseteq \sup^{A}_*(\H^A)+ \sup^{B}_*(\H^B)$. On the other hand, $\sup^{A}_*(\H^A), \sup^{B}_*(\H^B)\subseteq \sup^{X}_*(\H)$. Thus $\sup^{A}_*(\H^A)+ \sup^{B}_*(\H^B)\subseteq \sup^{X}_*(\H)$, and so $\sup^{X}_*(\H)= \sup^{A}_*(\H^A)+ \sup^{B}_*(\H^B)$

Now we show that $\inf^{A\cap B}_*(\H^{A\cap B})=\inf^{A}_*(\H^A)\cap \inf^{B}_*(\H^B).$ Clearly
$$
\inf^{A\cap B}_*(\H^{A\cap B})\subseteq \inf^{A}_*(\H^A)\cap \inf^{B}_*(\H^B).
$$
Conversely, note that
$$
\inf^{A}_*(\H^A)\cap \inf^{B}_*(\H^B)\subseteq (\Z(\H\cap A)\otimes G)\cap (\Z(\H\cap B)\otimes G)=\Z(\H\cap A\cap B)\otimes G
$$
where $(\Z(\H\cap A)\otimes G)\cap (\Z(\H\cap B)\otimes G)=\Z(\H\cap A\cap B)\otimes G$ because $C_*(X;G)$ is the direct sum of the copies of $G$ with its coordinates labeled by the graded set $X$. Thus $\inf^{A}_*(\H^A)\cap \inf^{B}_*(\H^B)$ is a subcomplex of $C_*(X;G)$ contained in $\Z(\H\cap A\cap B)\otimes G$, and so it is contained in $\inf^{A\cap B}_*(\H^{A\cap B})$.
\end{proof}

\begin{corollary}Using the notation as in the theorem above, the inclusion 
\[\xymatrix{
\inf^{A}_*(\H^A)\cap \inf^{B}_*(\H^B)\ \ar@{^{(}->}[r] & \sup^{A}_*(\H^A)\cap \sup^{B}_*(\H^B)}
\]
is a chain homotopy equivalence if and only if so is the inclusion 
\[\xymatrix{
\inf^{A}_*(\H^A)+ \inf^{B}_*(\H^B) \ \ar@{^{(}->}[r] & \sup^{A}_*(\H^A)+ \sup^{B}_*(\H^B).}
\]
\end{corollary}
\begin{proof}
The statement follows by applying the Five Lemma to the long exact sequence obtained from Theorem~\ref{thm:mv_seq_SHG}.
\end{proof}

\subsubsection{Gap Complexes}
Let $(\H,X)$ be a super-hypergraph and define
\begin{equation}\label{def:small_delta_H}
\delta^X(\H)=\bigcup\{ Y \subseteq\H \ | \ Y \textrm{ is a } \Delta-\textrm{subset of } X\}
\end{equation}
to be the largest $\Delta$-subset of $X$ contained in $\H$. Then $\delta^X(\H)$ consists of the elements in $\H$ whose all iterated faces lie in $\H$. The gap between $\delta^X(\H)\subseteq \Delta^X(\H)$ measures how far $\H$ is from being a $\Delta$-set. For a finite hypergraph $\H$, the differences can be expressed as
$$
\#(\Delta^X(\H)\smallsetminus\delta^X(\H))=\#(\Delta^X(\H))-\#(\delta^X(\H)).
$$

Topological invariants of the geometric gap complex
\begin{equation}\label{geometric gap complex}
|\Delta^X(\H)|/|\delta^X(\H)|
\end{equation}
such as homology groups and homotopy groups, provide different means of measuring how far $\H$ is from being a $\Delta$-set.

Algebraically, at the chain level
\begin{equation}
\xymatrix{
C_*(\delta^X(\H);G)\ar@{^{(}->}[r]& \inf_*^{C_*(\Delta^X;G)}(G(\H))\subseteq G(\H)\subseteq \sup_*^{C_*(\Delta^X;G)}(G(\H))\ar@{^{(}->}[r]&C_*(\Delta^X(\H);G)}
\end{equation}
where $G(\H)=\Z(\H)\otimes G$. An important consequence of Proposition~\ref{thm:inf->sup_inclus} is that the inclusion
$$
\xymatrix{
\inf_*^{C_*(\Delta^X;G)}(G(\H))\ar@{^{(}->}[r]& \sup_*^{C_*(\Delta^X;G)}(G(\H))}
$$
is a chain homotopy equivalence, 
which implies that the algebraic gap complex~(\ref{equation-gap}) is acyclic. However, the geometric gap complex~(\ref{geometric gap complex}) is not contractible in general. For example, let $X$ be any $\Delta$-set and let $\H$ be the graded subset of $X$ by removing the vertex set $X_0$. Then $\delta^X(\H)=\emptyset$ and so its geometric gap complex is $|X|^{+}$, the space $|X|$ disjoint union with a one-point set. The homology of the chain complexes
$$
\inf_*^{C_*(\Delta^X;G)}(G(\H))/C_*(\delta^X(\H);G),
$$
$$
C_*(\Delta^X(\H);G)/\sup_*^{C_*(\Delta^X;G)}(G(\H))
$$
could give extra information in addition to the topology of the geometric gap complex~(\ref{geometric gap complex}).

The proof of the following proposition is immediate.

\begin{proposition}
Let $(\H,X)$ be a super-hypergraph. Then
$$
\inf_*^{C_*(\Delta^X(\H);G)}(G(\H))/C_*(\delta^X(\H);G)=\inf^{C_*(\Delta^X(\H);G)/C_*(\delta^X(\H);G)}_*(G(\H)/C_*(\delta^X(\H);G)),
$$

$$
\sup_*^{C_*(\Delta^X(\H);G)}(G(\H))/C_*(\delta^X(\H);G)=\sup^{C_*(\Delta^X(\H);G)/C_*(\delta^X(\H);G)}_*(G(\H)/C_*(\delta^X(\H);G)).
$$

\hfill $\Box$
\end{proposition}

\subsubsection{Computations}

With the potential applications in mind, it is important to consider the computability of these topological constructions. There have been various algorithms developed for computing simplicial homology that have led to the applications of topology in data analytics. The computations of embedded homology are quite similar to those of simplicial homology. Below we detail a procedure for computing embedded homology.

For a super-hypergraph $(\H,X)$, let us consider the computations for $H_*^{\emb,X}(\H;\F)$ using the chain complex $\inf_*^{C_*(X;\F)}(\H;\F)$ with coefficients in a field $\F$.

By definition, $$H_n^{\emb,X}(\H,F)=Z_n(\inf_*^{C_*(X;\F)}(\H;\F))/B_n(\inf_*^{C_*(X;\F)}(\H;\F))$$  where
$$
Z_n(\inf_*^{C_*(X;\F)}(\H;\F))=\F(\H_n)\cap\Ker(\partial_n\colon C_n(X;\F)\to C_{n-1}(X;\F)),
$$
$$
B_n(\inf_*^{C_*(X;\F)}(\H;\F))=\F(\H_n)\cap \partial_{n+1}(C_{n+1}(X;\F)).
$$
The \textit{Betti number} $b_n(\H, X)$ is defined as
\begin{align}\label{equation}
b_n(\H, X)&=\dim H_n^{\emb,X}(\H;\F)\\
&=\dim (Z_n(\inf_*^{C_*(X;\F)}(\H;\F)) -\dim (B_n(\inf_*^{C_*(X;\F)}(\H;\F))) \nonumber.
\end{align}

To compute $Z_n(\inf_*^{C_*(X;\F)}(\H;\F))$, we can consider the restriction of the linear transformation $\partial_{n}\colon C_n(X;\F)\to C_{n-1}(X;\F)$ to $\F(\H_n)$. Namely, consider $\F(\H_n)$ as a vector spaces over $\F$ with a basis given by the elements in $\H_n$. For each element $x$ in $\H_n$, express $\partial_n(x)$ as an element in $C_{n-1}(X;\F)=\F(X_{n-1})$. This defines a linear transformation
$$
\partial_n|\colon \F(\H_n)\longrightarrow \F(X_{n-1})
$$
whose kernel is $Z_n(\inf_*^{C_*(X;\F)}(\H;\F))$.

For computing $B_n(\inf_*^{C_*(X;\F)}(\H;\F))$, we can first consider $\partial_{n+1}(C_{n+1}(X;\F))$ as a subspace of the vector space $C_n(X;\F)=\F(X_n)$, which is spanned by linear combinations of $\partial_{n+1}(\sigma)$ for $\sigma\in X_{n+1}$. Then consider the decomposition
$$
\F(X_n)=\F(\H_n)\oplus \F(X_n\smallsetminus \H_n).
$$
Let
$$
p\colon \F(X_n)\to \F(X_n\smallsetminus \H_n)
$$
be the projection. Then $B_n(\inf_*^{C_*(X;\F)}(\H;\F))$ is the kernel of the restriction
$$
p|\colon \partial_{n+1}(C_{n+1}(X;\F))\longrightarrow \F(X_n\smallsetminus \H_n).
$$

If the data $(\H,X)$ is large, the complexity of  direct computation of $H_*^{\emb,X}(\H;\F)$ increases. The existing computational methods for chain complexes aim at reducing this complexity.

\section{Super-Persistent Homology}\label{section4}
The general idea of super-persistent homology is to use the geometry of $\Delta$-sets and super-hypergraphs as tools to investigate collections of subgraphs in a given graph and to get topological features from for example graphic data and various networks. We will see that a $\Delta$-set model performs much better than models using simplicial complexes, particularly for exploring topological features arising from the structures related to clustering.

\subsection{General Theory}

Let $G$ be a directed/undirected (multi)-graph. Let $\mathcal{FP}(G)$ denote the set of all finite subgraphs of $G$.
\begin{definition}
A $\Delta$-set $X$ is said to be \textit{dominated by $G$} if there exists an injective map $\phi\colon X\longrightarrow \mathcal{FP}(G)$ such that $\phi(d_i(\sigma))$ is a subgraph of $\phi(\sigma)$ for any $0\leq i\leq n$ and any element $\sigma\in X_n$.

A super-hypergraph $(\H, X)$ is said to be \textit{dominated by $G$} if its parental $\Delta$-set $X$ is dominated by $G$.
\end{definition}

For a $\Delta$-set $X$ dominated by $G$, we identify the elements $\sigma$ in $X$ with its image $\phi(\sigma)$, a finite subgraph of $G$, and so we consider $X$ as a collection of finite subgraphs of $G$. A super-hypergraph $(\H,X)$ dominated by $G$ can be described as a multi-layered collection of families of finite subgraphs $X=\{X_0,X_1,\ldots\}$, where each $X_i$ is a family of finite subgraphs, with face operations $d_i\colon X_n\to X_{n-1}$, $0\leq i\leq n$, satisfying the $\Delta$-identity, and a marked graded subset $\H$ of $X$.

To introduce persistence we need a scoring scheme on  $G$. For graphs $P$ and $Q$, denote by $P\preceq Q$ if $P$ is a subgraph of $Q$.

\begin{definition}
A \textit{scoring scheme} on a directed/undirected (multi-)graph $G$ is a function
$$
\mathfrak{M}\colon \mathcal{FP}(G)\longrightarrow \R
$$
from the set of finite subgraphs of $G$ to the real numbers.

A scoring scheme $\mathfrak{M}$ on $G$ is called \textit{regular} if for every $ P,Q \in \mathcal{FP}(G)$ such that $P\preceq Q$,
$$
\mathfrak{M}(P)\leq \mathfrak{M}(Q).
$$
\end{definition}

\begin{definition}
A \textit{persistent $\Delta$-filtration} of a $\Delta$-set $X$ over $\R$ is a family of $\Delta$-subsets $X(t)$ of $X$, indexed by $t\in \R$, such that
\begin{enumerate}
\item[1)] $X(s)$ is a $\Delta$-subset of $X(t)$ for $s\leq t$, and
\item[2)] $X=\bigcup\limits_{t\in\R}X(t)$.
\end{enumerate}

A \textit{persistent super-hypergraph filtration} of a super-hypergraph $(\H,X)$ over $\R$ is a family of super-hypergraphs $(\H(t),X(t))$, indexed by $t\in\R$, such that
\begin{enumerate}
\item[1)] the indexed family $X(t)$, $t\in\R$, is a persistent $\Delta$-filtration of $X$ over $\R$,
\item[2)] $\H(t)=\H\cap X(t)$.
\end{enumerate}
\end{definition}

\begin{proposition}
Let $G$ be a directed/undirected (multi-)graph with a regular scoring scheme $\mathfrak{M}\colon \mathcal{FP}(G)\longrightarrow \R$. Let $(\H,X)$ be a super-hypergraph dominated by $G$. Then
$$
X(t)=\mathfrak{M}^{-1}((-\infty,t])\cap X, \quad t\in\R
$$
is a persistent $\Delta$-filtration of $X$, and the pair
$$
(\H(t),X(t))=(\mathfrak{M}^{-1}((-\infty,t])\cap \H,\mathfrak{M}^{-1}((-\infty,t])\cap X), \quad a\in\R
$$
is a persistent super-hypergraph filtration of $(\H,X)$.
\end{proposition}
\begin{proof}
The proof follows from the definitions.
\end{proof}

Let $\F$ be a fixed choice of a ground field. A \textit{(graded/ungraded) persistence module} over $\F$, denoted by $\mathbb{V}$, is defined to be an indexed family of (graded/ungraded) $\F$ vector spaces
$$
\mathbb{V} = (V(t)\ |\ t \in \R)
$$
and a bi-indexed family of (graded/ungraded) linear maps
$$
(v_s^t\colon V(s) \to V(t)\ |\ s \leq t)
$$
which satisfy the composition law
$$
v_s^t\circ v_r^s= v_r^t
$$
whenever $r\leq s \leq t$, where $v^t_t$ is the identity map on $V(t)$. For graded vector spaces $V$ and $W$, a \textit{graded linear map} of degree $q$ is a collection of linear maps $\phi=(\phi_n)_{n\in\Z}$ with $\phi_n\colon V_n\to W_{n+q}$. A \textit{persistence morphism $\Phi$ of dregree $q$} between two graded persistence modules $\mathbb{V}$ and $\mathbb{W}$ is a collection of graded linear maps of degree $q$, $(\phi^t\colon V(t)\to W(t) \ | \ t\in\R)$, such that the diagram
$$
\xymatrix{
V(s)\ar@{->}[r]^{v_s^t}\ar@{->}[d]^{\phi^s}&V(t)\ar@{->}[d]^{\phi^t}\\
W(s)\ar@{->}[r]^{w_s^t}&W(t)\\}
$$
commutes for $s\leq t$. If $q=0$, then $\Phi$ is called a \textit{persistence morphism}. A persistence morphism between ungraded persistence modules is defined in the same way.

\begin{definition}
Let $(\H,X)$ be a super-hypergraph. Let $A$ be an abelian group. The \textit{relative embedded homology $H_*^{\emb,X}(X,\H;A)$ with coefficients in $A$} of $(\H,X)$ is defined by
 $$
 \begin{array}{rcl}
 H_*^{\emb,X}(X,\H;A)&=&H_*(C_*(X;A)/\inf^{C_*(X;A)}_*(\Z(\H)\otimes A))\\
 &\cong& H_*(C_*(X;A)/\sup^{C_*(X;A)}_*(\Z(\H)\otimes A)).\\
 \end{array}
 $$
\end{definition}

Now we assume that all homology is taken with coefficients in a ground field $\F$. Therefore, we can simplify our notation of homology groups $H_*(\ - \ ; \F)$ to $H_*(\ - \ )$. Let $(\H,X)$ be a super-hypergraph dominated by a directed/undirected (multi-)graph $G$ with a scoring scheme $\mathfrak{M}$. Let $X(t)=\mathfrak{M}^{-1}((-\infty,t])\cap X$ and $\H(t)=\mathfrak{M}^{-1}((-\infty,t])\cap \H$. Then
\begin{equation}\label{super-persistent homology}
\begin{array}{cc}
\mathbb{H}_*(X)=(H_*(X(t))\ | \ t\in\R),&\\
\mathbb{H}_*^{\emb,X}(\H)=(H_*^{\emb,X}(\H(t))\ | \ t\in\R) &\textrm{ and }\\
\mathbb{H}_*^{\emb,X}(X,\H)=(H_*^{\emb,X}(X,\H(t))\ | \ t\in\R)&\\
\end{array}
\end{equation}
are graded persistence modules.
The short exact sequence of chain complexes
\[
\xymatrix{
\inf^{C_*(X(t))}_*(\F(\H(t)))\ar@{^{(}->}[r]^-{j^t}& C_*(X(t))\ar@{->>}^-{p^t}[r]&  C_*(X(t))/\inf^{C_*(X(t))}_*(\F(\H(t)))}
\]
induces a long exact sequence on homology, which can be written as an exact triangle of graded persistence modules
\begin{equation}\label{super-persistent triangle}
\xymatrix{
\mathbb{H}_*^{\emb,X}(\H)\ar@{->}[rr]^-{\mathbb{J}} & &\mathbb{H}_*(X)\ar@{->}[dl]^-{\mathbb{P}}\\
&\mathbb{H}_*^{\emb,X}(X,\H)\ar@{->}[ul]^{\partial}&\\}
\end{equation}
where $\mathbb{J}$ and $\mathbb{P}$ are persistence morphisms and the boundary homomorphism $\partial$ is persistence morphism of degree $-1$.

\begin{definition}
Let $(\H,X)$ be a super-hypergraph dominated by a directed/undirected (multi-)graph $G$ with a scoring scheme. Then the three graded persistence modules listed in~(\ref{super-persistent homology}) together with the exact triangle~(\ref{super-persistent triangle}) give a \textit{super-persistent homology} of $(\H,X)$ with coefficients in a field $\F$.
\end{definition}

Equally, we could call $\mathbb{H}_*^{\emb,X}(\H)$ super-persistent homology. However, the definition given above can carry more information than the embedded homology of the super-hypergraph, as we will now illustrate.

Let $J\subseteq \R$ be an interval. The \textit{interval persistence module} $\F^J=(J(t) \ | \ t\in\R)$ is defined by
$$
J(t)=\left\{
\begin{array}{ll}
\F&\textrm{if } t\in J,\\
0&\textrm{otherwise}\\
\end{array}\right.
$$
with double indexed linear maps
$$
j^t_s=\left\{
\begin{array}{ll}
\mathrm{id}&\textrm{ if } s,t\in J,\\
0&\textrm{ otherwise.}\\
\end{array}\right.
$$
For an ungraded persistence module $\mathbb{V}$, the $q$-th \textit{suspension} $\Sigma^q\mathbb{V}$ is a graded persistence module with
$$
\Sigma^q\mathbb{V}_n=\left\{
\begin{array}{ll}
\mathbb{V}&\textrm{ if } n=q\\
0&\textrm{ otherwise.}\\
\end{array}\right.
$$
A \textit{graded interval persistence module} is a $q$-th suspension of the ungraded interval persistence module for some $q$. 

Recall that a $\Delta$-set $X$ is called of \textit{finite type} if $X_n$ is finite for each $n$.

\begin{theorem}[Structure Theorem]\label{thm:SPHom_structure}
Let $(\H,X)$ be a super-hypergraph dominated by a directed/undirected (multi-)graph $G$ with a scoring scheme. Suppose that $X$ is of finite type. Then the graded persistence modules $\mathbb{H}_*(X)$, $\mathbb{H}_*^{\emb,X}(\H)$ and $\mathbb{H}_*^{\emb,X}(X,\H)$ admit direct sum decompositions in terms of graded interval persistence modules and these decompositions are unique up to the order of factors in the category of graded persistence modules.
\end{theorem}
\begin{proof}
For any graded persistence module $\mathbb{V}$, there is a canonical decomposition
$$
\mathbb{V}\cong \bigoplus_{n\in\Z}\Sigma^n\mathbb{V}_n
$$
in the category of graded persistence modules, where $\mathbb{V}_n$ is considered as an ungraded persistence module. It suffices to show that $\mathbb{H}_n(X)$, $\mathbb{H}_n^{\emb,X}(\H)$ and $\mathbb{H}_n^{\emb,X}(X,\H)$ admit unique factorization as ungraded persistence modules for each $n$.

Since $X$ is of finite type, the chain complex $C_*(X)$ is of finite type and so is any subcomplex or quotient complex. The assertion follows from the structure theorem on persistence modules~\cite[Theorem 1.1]{Crawley-Beovey} derived from the classical Gabriel Theorem in representation theory~\cite{Gabriel}.
\end{proof}

We now briefly summarise persistence diagrams/persistent barcodes; for which this structure theorem is prominent in the calculations, for more details see~\cite{book-persistence-module}.

Let $\mathbb{V}$ be an ungraded persistence module with a unique decomposition (up to the order of factors)
$$
\mathbb{V}=\bigoplus_{\alpha\in I}\F^{J_{\alpha}}
$$
in terms of interval persistence modules, where $I$ is the index set. Then the multi-set given by $(\inf(J_{\alpha}),\sup(J_{\alpha}))\subset \R^2$, $\alpha\in I$, is the \textit{persistence diagram} (or \textit{barcode}) of $\mathbb{V}$, denoted by $\mathrm{dgm}(\mathbb{V})$. In our case, under the hypothesis that $X$ is of finite type, we have three persistence diagrams.
\begin{corollary}
Let $(\H,X)$ be a super-hypergraph dominated by a directed/undirected (multi-)graph $G$ with a scoring scheme and let $X$ be of finite type. Then there are three multi-layer persistence diagrams
$\mathrm{dgm}(\mathbb{H}_n(X))$, $\mathrm{dgm}(\mathbb{H}_n^{\emb,X}(\H))$ and $\mathrm{dgm}(\mathbb{H}_n^{\emb,X}(X,\H))$ for $n\geq0$. \hfill $\Box$
\end{corollary}

The exact triangle~(\ref{super-persistent triangle}) yields matrix data on the correlations of the multi-layer persistence diagrams as follows. Let 
$$\Phi\colon \mathbb{V}\longrightarrow \mathbb{W}$$
be a persistence morphism between ungraded persistence modules. Suppose that both $\mathbb{V}$ and $\mathbb{W}$ satisfy the unique factorization property with respect to decompositions in terms of interval persistence modules and let
$$
\mathbb{V}=\bigoplus_{\alpha\in I_V}\F^{J_{\alpha}}\ \textrm{ and } \mathbb{W}=\bigoplus_{\beta\in I_W}\F^{J_{\beta}}.
$$
Let $\Phi_{\alpha,\beta}$ be the composite
$$
\xymatrix{
\Phi_{\alpha,\beta}\colon \F^{J_{\alpha}}\ar@{->}[rr]^-{\Phi|_{\F^{J_{\alpha}}}}&&\mathbb{W}\ar@{->>}[rr]^-{\mathrm{proj}}&&\F^{J_{\beta}}\\}.
$$
According to ~\cite[Proposition 16, Lemma 22] {Bubenik}, the set of persistence morphisms between two interval persistence modules is either a 1-dimensional vector space or $0$. Thus $\Phi_{\alpha,\beta}\colon \F^{J_{\alpha}}\to \F^{J_{\beta}}$ is either a generator for $\Hom(\F^{J_{\alpha}},\F^{J_{\beta}})$, or zero. Let the index sets $I_V$ and $I_W$ be totally ordered and define the \textit{correlation matrix}
$$M(\Phi)=(m_{\alpha,\beta})_{\alpha\in I_V,\beta\in I_W}$$ of $\Phi$ by setting
$$
m_{\alpha,\beta}=\left\{
\begin{array}{ll}
1&\textrm{ if } \Phi_{\alpha,\beta}\not=0,\\
0&\textrm{ otherwise.}\\
\end{array}\right.
$$
The correlation matrix is an analogue of adjacency matrix of graphs, which gives correlations between $\mathrm{dgm}(\mathbb{V})$ and $\mathrm{dgm}(\mathbb{W})$ by adding directed edges. In summary, we have the following information data from super-persistent homology.

\begin{proposition}
Let $(\H,X)$ be a super-hypergraph dominated by a directed/undirected (multi-)graph $G$ with a scoring scheme. Suppose that $X$ is of finite type. Then there are three multi-layer persistence diagrams $\mathrm{dgm}(\mathbb{H}_*(X))$, $\mathrm{dgm}(\mathbb{H}_*^{\emb,X}(\H))$ and $\mathrm{dgm}(\mathbb{H}_*^{\emb,X}(X,\H))$ together with three correlation matrices $M(\mathbb{J})$, $M(\mathbb{P})$ and $M(\partial)$ between them.\hfill $\Box$
\end{proposition}

An important point is that we allow $\H$ to be an arbitrary graded subset of $X$. If we fix $X$ to be a $\Delta$-set dominated by a graph $G$ with a scoring scheme and allow $\H$ to be random, then $\mathrm{dgm}(\mathbb{H}_*(X))$ is a deterministic barcode, while $\mathrm{dgm}(\mathbb{H}_*^{\emb,X}(\H))$ and $\mathrm{dgm}(\mathbb{H}_*^{\emb,X}(X,\H))$ are random. The correlation matrices may help further analyse the data.

Let $X$ be a fixed $\Delta$-set dominated by a directed/undirected (multi-)graph. It would be also interesting to consider the set $\mathbb{P}(X)$ of the isomorphic classes of persistence modules~(\ref{super-persistent homology}) for all graded subsets $\H$ of $X$. The inclusions $\H'\subseteq \H\subseteq X$ induce a morphism of super-hypergraphs $(\H',X)\to (\H,X)$. By taking super-persistent homology, one would get a quiver structure on the set $\mathbb{P}(X)$. Moreover, there is an \textit{interleaving distance} between persistence modules introduced in~\cite{Chazal2009} corresponding to \textit{bottleneck distance}~\cite{book-persistence-module}, which gives the structure of a metric space on the quiver $\mathbb{P}(X)$. The following example illustrates that $\mathbb{P}(X)$ may give more robust information.

\begin{example}
Let $X$ be a $\Delta$-set. Let $\H$ be a graded subset of $X$ consisting of \textit{non-face} elements $\sigma\in X$. Here a non-face element $\sigma$ means that there does not exist any element $\tau\in X$ such that $d_i\tau=\sigma$ for some $i$. In other words, $\H$ is given by removing all face elements in $X$. It is straightforward to see that
$$
H_*^{\emb,X}(\H)=\F(\H)\cap \Z(C_*(X)).
$$
Therefore, the embedded homology can detect the cycles contributed from non-face elements. From the exact triangle~(\ref{super-persistent triangle}), a part of the boundaries in $C_*(X)$ can also be detected. These detected elements would contribute to bar-codes through persistence.  Hence, in addition to the persistence on the homology $H_*(X)$, $\mathbb{P}(X)$ can decode more topological features.
\end{example}

\subsection{The Ordinary Persistent Homology}\label{ordinary persistent homology}
The classical persistent homology refers to the persistent homology of point cloud data using the Vietoris-Rips complex, the \v{C}ech complex or the witness complex. Persistent homology has been used as an important topological tool in data science. In this subsection, we rewrite on the classical persistent homology from the viewpoint of graphs with scoring schemes, and then give a natural generalization to persistent homology for graphs with reference maps to metric spaces.

\subsubsection{The Classical Persistent Homology}
A \textit{point cloud} dataset is a finite set $\mathcal{L}$ with a reference map that embeds $\mathcal{L}$ into a finite dimensional Euclidean space $\R^m$, thus we can consider $\mathcal{L}$ as a finite subset in $\R^m$. We now use a graph with a scoring scheme to describe the (persistent) Vietoris-Rips complex, \v{C}ech complex and witness complex in a unified way. The graph $G(\mathcal{L})$ is the complete graph having $\mathcal{L}$ as its vertex set. Intuitively, we assign one and only one edge to any two distinct points in $\mathcal{L}$. The main point is to show that different scoring schemes on $G(\mathcal{L})$ can obtain different persistent complexes that are currently widely used in TDA~\cite{Carlsson}. Below we give the scoring schemes for the Vietoris-Rips complex, the \v{C}ech complex, the strong witness complex and the weak witness complex. Let $\Lambda=\{l_0,l_1,\ldots,l_n\}\subseteq \mathcal{L}$ be a subset of $\mathcal{L}$.

For $x\in \R^m$, denote by $B(x,r)$ the closed ball of radius $r$ centered at $x$.
Define the following scoreing schemes on $\Lambda$:
\begin{enumerate}
\item[1)] The \textit{Vietoris-Rips scoring} is given by
\begin{equation}\label{Vietoris-Rips scoring}
\mathfrak{M}^{VR}(\Lambda)=\frac{1}{2}\sup\{d(l_i,l_j) \ | \ l_i,l_j\in\Lambda\subseteq \R^m\}.
\end{equation}
The balls $B(l_0,r),\ldots,B(l_n,r)$ pairwise intersect if and only if the score $\mathfrak{M}^{VR}(\Lambda)\leq r$.

\item[2)] The \textit{\v{C}ech scoring} is defined by
\begin{equation}\label{Cech scoring} 
\mathfrak{M}^{\check{C}}(\Lambda)=\inf_{x\in\R^m}\max\{d(x,l)\ | \ l\in \Lambda\subseteq \R^m\}.
\end{equation}
Note that
$$
\bigcap_{i=0}^n B(l_i,r)\not=\emptyset
$$
if and only if the score $\mathfrak{M}^{\check{C}}(\Lambda)\leq r$.

\item[3)] The \textit{Strong witness scoring} is defined by
\begin{equation}\label{Strong witness scoring}
\mathfrak{M}^{W^s}(\Lambda)=\inf_{x\in\R^m}\{\sup_{y\in\Lambda}d(x,y)-\inf_{z\in \mathcal{L}}d(x,z)\}.
\end{equation}
\item[4)] The \textit{Vietoris-Rips Strong witness scoring} is given by
\begin{equation}\label{Vietoris-Rips Strong witness scoring}
\mathfrak{M}^{W^s_{VR}}(\Lambda)=\sup_{0\leq i<j\leq n}\{\inf_{x\in\R^m}\{\max\{d(x,l_i),d(x,l_j)\}-\inf_{z\in \mathcal{L}}d(x,z)\}\}.
\end{equation}
\item[5)] The \textit{Weak witness scoring} is given by
\begin{equation}\label{Weak witness scoring}
\mathfrak{M}^{W^w}(\Lambda)=\inf_{x\in\R^m}\{\sup_{y\in\Lambda}d(x,y)-\inf_{z\in \mathcal{L}\smallsetminus \Lambda}d(x,z)\}.
\end{equation}
\item[6)] The \textit{Vietoris-Rips weak witness scoring} is given by
\begin{equation}\label{Vietoris-Rips weak witness scoring}
\mathfrak{M}^{W^w_{VR}}(\Lambda)=\sup_{0\leq i<j\leq n}\{\inf_{x\in\R^m}\{\max\{d(x,l_i),d(x,l_j)\}-\inf_{z\in \mathcal{L}\smallsetminus\Lambda}d(x,z)\}\}.
\end{equation}
\end{enumerate}

Let $X$ be the clique complex of $G(\mathcal{L})$ considered as a $\Delta$-set. In other words, because  $G$ is complete, $X$ is the set of all nonempty full subgraphs of $G(\mathcal{L})$. Here a \textit{full subgraph} $H$ of $G$ is a subgraph $H$ such that the edge set between any two vertices $v$ and $w$ in $H$ is equal to the edge set between $v$ and $w$ in $G$. Choose a linear order on the vertex set $V(G(\mathcal{L}))$. Then $(X,X)$ is a super-hypergraph. It is straightforward to check that the persistent super-hypergraph filtrations on $(X,X)$ induced by the above scoring schemes coincide with the classical persistent filtrations in~\cite{Carlsson}. Here the witness scoring schemes defined in (3)-(6) are reformulations from~\cite[Definition~2.7, Definition 2.8]{Carlsson}.

\subsubsection{Clique Persistent Homology of Graphs with Reference Maps}

Next, we will consider a canonical extension of ordinary persistent homology to the case of graphs with reference maps on vertices. Let $G$ be a finite undirected (multi-)graph with a reference map that embeds the vertex set $V(G)$ into a finite dimensional Euclidean space $\R^m$ and let $H$ be any subgraph. Then, any one of the six scoring schemes in~(\ref{Vietoris-Rips scoring})-(\ref{Vietoris-Rips weak witness scoring}) induces a persistent filtration on the clique complex $\mathrm{clique(G)}$. This gives the \textit{clique persistent homology} on $G$.

The clique persistent homology on $G$ could, in general, be quite different from the ordinary persistent homology of the vertex set of $G$ under the reference map. For instance, if the clique complex $\mathrm{clique}(G)$ has non-trivial reduced homology, the resulting clique persistent homology converges to $H_*(\mathrm{clique}(G))$ as $t\to\infty$, but the ordinary persistent homology converges to trivial homology as $t\to\infty$. The ordinary persistent homology of $V(G)$ under the reference map is obtained by rebuilding a new graph given by the complete graph on $V(G)$, that is, all edges in $G$ are forgotten and the new edges are added in depending on the scoring schemes that are obtained through the reference map. On the other hand, when we consider the clique complex of $G$ itself, the edges in $G$ are accounted for.

The following example illustrates how clique persistent homology can describe shapes and therefore could be useful for data analysis on protein structures or image processing on 3D objects with complicated internal structures such as the heart.

\begin{example}
Let $X$ be a polyhedron in $\R^m$. Let $K$ be a simplicial complex that is a triangulation of $X$ and let $G$ be the graph given by the $1$-skeleton of the barycentric subdivision of $K$. Let the reference map on $G$ be given by the inclusion of $G$ in $\R^m$. Then the geometric realization of the clique complex $\mathrm{clique}(G)$ is homeomorphic to $X$, and the persistent homology of $\mathrm{clique}(G)$ converges to $H_*(X)$. In particular, the number of infinite persistence modules in the $n^\text{th}$ persistent homology of $\mathrm{clique}(G)$ is equal to the $n^\text{th}$ Betti number of $X$. Thus the clique persistent homology detects the topological shape of $X$.
\end{example}

More generally we can remove the embedding hypothesis of reference maps. Let $G$ be a finite undirected (multi-)graph and let
$$
f\colon V(G)\longrightarrow \R^m
$$
be a function (without assuming injectivity). We can use \textit{pull-back scoring} in the following sense. Let $H$ be any subgraph of $G$. Then the image $f(V(H))$ is a finite subset located in $\R^m$. Let $\mathfrak{M}$ be a scoring scheme on point cloud data such as one of the six aforementioned scoring schemes. Define
\begin{equation}
\begin{array}{lr}
\mathfrak{M}^f(H)=\mathfrak{M}(f(V(H)))\qquad &\textrm{(Pull-back Scoring)}\\
\end{array}
\end{equation}
which induces a persistent filtration on the clique complex $\mathrm{clique}(G)$ depending on the reference map $f$. Different choices of $f$ would result in different persistence diagrams. For instance, a constant function does induce a trivial persistence on $H_*(\mathrm{clique}(G))$. The flexibility of $f$ could be useful. For example, if $f$ is randomly given, it induces corresponding random persistence diagram.

The following example illustrates that a pull-back scoring on clique persistent homology may be useful for detecting higher dimensional geometric shapes.
\begin{example}
Let $p\colon E\to B$ be a continuous map between polyhedra $E$ and $B$. Assume that $B$ is a subspace of $\R^m$. By triangulating $B$, we pull it back along $p$ do define $K^E$ and there is a simplicial map $p'\colon K^E\longrightarrow K^B$ such that
\begin{enumerate}
\item[1)] $K^E$ and $K^B$ are simplicial complexes of triangulations of $E$ and $B$, respectively.
\item[2)] Let $G(K^E)$ and $G(K^B)$ be the graphs given by the $1$-skeleton of $K^E$ and $K^B$, respectively. Then $\mathrm{clique}(G(K^E))=K^E$ and $\mathrm{clique}(G(K^B))=K^B$.
\item[3)] $p'$ is a simplicial approximation to $p$.
\end{enumerate}
Let the reference map $R$ on $V(G(K^B))$ be given by the inclusion $V(G(K^B))\subseteq B\subseteq \R^m$. Let us take the pull-back scoring $\mathfrak{M}$ on $G(K^E))$. Then the persistent filtration on $\mathrm{clique}(G(K^E))$ induced by $\mathfrak{M}$ is the pull-back of the persistent filtration on $\mathrm{Clique}(G(K^B))$ induced by $R$. In particular, if $p\colon E\to B$ is a fibre bundle or, more generally, a fibration with fibre $F$, then we have a persistent Leray-Serre spectral sequence convergent to $H_*(E)$. It is well-known in algebraic topology that Leray-Serre spectral sequences are an important tool for computing $H_*(E)$ starting with $H_*(B)$ and $H_*(F)$.
\end{example}

\subsection{Partition Homology and Persistent Partition Homology}\label{clustering}
The methods of data science are typically aimed at finding structures and patterns within large datasets. Being able to glean information about the internal structures of graphical data would be useful in solving the typical problems given to machine learning algorithms. For example, classification problems, prediction and, in particular, partitioning data into clusters ~\cite[Section 1.1.3]{RRC}.

If a collection of subgraphs forms a $\Delta$-set structure, then we can calculate homology. A natural question is how to introduce a $\Delta$-set structure on a collection of subgraphs in some natural way. More precisely, how to define face operations on subgraphs. We are going to show that any clustering on the vertex set can induce canonical face operations on subgraphs. For a dataset given by a graph, the topological features on collections of subgraphs under the face operations induced by a clustering may help for detecting correlations between the clusters.

Let $G$ be a directed/undirected (multi-)graph. Assume that there is a disjoint clustering $\bfp$ on the vertex set $V(G)$. In other words, there is a disjoint union
$$
V(G)=\coprod_{i=0}^m V_i(G)
$$
under the clustering $\bfp$, where each $V_i(G)$ is a cluster. Let $H$ be a subgraph of $G$. Then there exists a unique sequence $(k_0,k_1,\ldots,k_n)$ with $0\leq k_0<k_1<\ldots<k_n\leq m$ such that $V(H)\cap V_i(G)\not=\emptyset$ for $i\in\{k_0,k_1,\ldots,k_n\}$ and $V(H)\cap V_i(G)=\emptyset$ if $i\not\in\{k_0,k_1,\ldots,k_n\}$. We call $H$ a subgraph of $G$ \textit{linked to $(n+1)$ clusters}. Viewing $H$ as an abstract $n$-simplex, we define the $j^{\text{th}}$-face map, $d_j^\bfp(H)$, as the full subgraph of $H$ formed after removing all of those vertices $v\in V(H)\cap V_{k_j}(G)$ together with the edges incident  to (or from) such $v$. The resulting subgraph $d_j^\bfp(H)$ is linked to $n$ clusters with $V(d_j^\bfp(H))\cap V_{k_j}(G)=\emptyset$. It is straightforward to show that the $\Delta$-identity
$$
d_i^\bfp d_j^\bfp=d_j^\bfp d^\bfp_{i+1}
$$
for $i\geq j$ holds. The face operation $d_j^\bfp$ is \textit{induced by the disjoint clustering $\bfp$}.

Now let $\H$ be any collection of subgraphs of $G$. Define $\H_n$ to be the subset of $\H$ consisting of those subgraphs in $\H$ linked to $(n+1)$ clusters. This gives a graded structure on $\H=\{\H_n\}_{n\geq0}$. Let $X(\H)$ be the collection of subgraphs of $G$ given by all of the elements in $\H$ together with all iterated faces. Then $X(\H)$ is a $\Delta$-set and $(\H,X(\H))$ is a super-hypergraph. The resulting homology groups
$$
H_*(X), H_*^{\emb,X}(\H), H_*^{\emb,X}(X,\H)
$$
are called the \textit{partition homology}.

In general, $X(\H)$ may not be a simplicial complex. In simplicial complexes, the assumption that simplices are determined by their vertices is too strict with applications in mind. For exploring topological structures arising from disjoint clusterings, a $\Delta$-set is a more suitable notion.

In practice, for studying possible correlations between clusters, one could start with a collection of one or more subgraphs, $\H$, linked with some or all clusters, and then produce the $\Delta$-set $X(\H)$. Due to the nature of simplicial homology, higher dimensional homology of $X$, and higher dimensional embedded homology of $(\H,X)$ would give topological features measuring the group correlations between more clusters. In particular, we have set up the $\Delta$-set $X(\H)$ such that the homological dimension of a given subgraph $H$ in $X$ is $n-1$, where $n$ is the number of clusters linked with $H$.

In theory, we can start with any collection of subgraphs as the initial data $\H$. For instance, we can start with $\H$ given by all or some of the $k$-regular subgraphs of $G$, Eulerian subgraphs, traceable subgraphs or Hamitonian subgraphs as initial data for constructing the $\Delta$-set $X(\H)$. This would give different topological approaches to understanding the internal structures of $\H$ in addition to what we have discussed previously in Section~\ref{sec:topstructures}.

If there is a scoring scheme on $G$, then we can calculate the super-persistent homology of $(\H,X)$ called \textit{persistent partition homology}. A scoring scheme on $G$ can be deterministically or randomly given. If a scoring scheme is randomly given, it may not be regular. This means that in the induced persistent filtration on $X$, the graded subset $X(t)$ may not be a $\Delta$-subset of $X$ for all $t$. In this case, the persistence system can be modified replacing chains related to the terms $X(t)$ by infimum or supremum chains on $X(t)$. The resulting persistence modules and persistence diagrams can be modified accordingly.

From the  perspective of data processing,  a clustering may be compared against certain optimization properties. Currently, we use discrete Morse theory which works well on simplicial complexes and cell complexes~\cite{Forman}, and chain complexes~\cite{Kozlov2005} with applications in data analysis~\cite{Reininghaus}. Moreover, the combinatorial Laplacian operator works well on simplicial complexes and chain complexes, where cohomology with coefficients in real numbers can be expressed as the null space of the Laplacian operator on cochains~\cite{HorakandJost}.

\subsection{Other Face Operations}
We have shown how disjoint clusterings can induce face operations on subgraphs. This construction works well as a theory which unifies many constructions such as the clique complex and the neighborhood complex. However, if we are interested in subgraphs having some special properties, this construction has some disadvantages. For instance, if we are interested in collections of finite connected subgraphs $H$, then the subgraph of $H$ given by removing some of its vertices may not be connected. In the following subsection, we will discuss some alternative ways of getting natural constructions of face operations.

\subsubsection{Link-blowup Face Operations}

The idea of link-blowup face operations comes from geometric constructions on tubular neighborhoods of submanifolds or regular neighborhoods of subcomplexes. Let $H$ be a subgraph of some graph $G$ and let $S$ be a subset of the vertex set $V(H)$. If we remove $S$ from $H$, then we could add some edges from the working graph $G$ to make a blowup for the subgraph $H\smallsetminus S$. A natural way to add these edges is to consider the neighbors of vertices of $S$ in $H$ and to add the edges between these neighbors from the working graph $G$.

Let $G$ be a directed/undirected (multi-)graph and let $E^G(v,w)$ denote the edge set between $v$ and $w$ for vertices $v,w\in V(G)$. Let $S\subseteq V(G)$ be a subset and define the \textit{induced subgraph of $S$ in $G$}, $G[S]$, to be the full subgraph of $G$ having $S$ as its vertex set. The \textit{closed neighborhood set} $N[S]$ is defined by
$$
N[S]=S\cup \{u\ | \ u\in V(G) \textrm{ adjacent to a vertex } v\in S\}
$$
namely, $N[S]$ is the union of the neighborhoods of vertices $v\in S$. The \textit{link set}  of $S$ in $G$ is
$$
\lk(S)=N[S]\smallsetminus S.
$$

Let $H$ be a subgraph of $G$ and let $S\subseteq V(H)$ be a subset of the vertex set of $H$. The \textit{link-blowup}\footnote{This definition is taken from a geometric setting. We may consider the subgraph $G[N[S]]$ as a regular neighborhood of $S$. Then $\mathrm{Lk}(S)$ are the vertices located in the ``\textit{boundary}'' of the regular neighborhood. Geometrically, we add all of edges in $G$ joining vertices in $\mathrm{Lk}(S)\cap V(H)$ to form a blowup on $H[V(H)\smallsetminus S]$. } of $H$ along $S$ is defined as
$$
H[V(H)\smallsetminus S]\cup G[\lk(S)\cap V(H)].
$$
Note that the graph $H[V(H)\smallsetminus S]\cup G[\lk(S)\cap V(H)]$ has the same vertex set of $H[V(H)\smallsetminus S]$.

Now suppose that there is a disjoint clustering $\bfp$ on the vertex set $V(G)$ so that there is a disjoint union
$$
V(G)=\coprod_{i=0}^m V_i(G)
$$
under $\bfp$ with each $V_i(G)$ a cluster. Let $H$ be a subgraph. Then there exists a unique sequence $(k_0,k_1,\ldots,k_n)$ with $0\leq k_0<k_1<\ldots<k_n\leq m$ such that $V(H)\cap V_i(G)\not=\emptyset$ for $i\in\{k_0,k_1,\ldots,k_n\}$ and $V(H)\cap V_i(G)=\emptyset$ if $i\not\in\{k_0,k_1,\ldots,k_n\}$. Let $V_j(H)=V(H)\cap V_{k_j}(G)$. Define the \textit{link-blowup face operation} as
\begin{equation}\label{link-blowup face}
d_j^{\mathrm{lk}}(H)=H[V(H)\smallsetminus V_j(H)]\cup G[\lk(V_j(H))\cap V(H)]
\end{equation}
that is, the link-blowup of $H$ along $V_j(H)=V(H)\cap V_{k_j}(G)$ for $0\leq j\leq n$.

\begin{proposition}\label{link-blowup face identity}
Given a disjoint clustering $\bfp$, let $d_j^{\slk}$ be defined as above. Then
\begin{equation}
d_i^{\mathrm{lk}}d_j^{\mathrm{lk}}=d_j^{\mathrm{lk}}d_{i+1}^{\mathrm{lk}}
\end{equation}
for $i\geq j$.
\end{proposition}
\begin{proof}
Let $H$ be a subgraph of $G$. Let $V_j$ denote $V_j(H)$ defined as above and let $d_j$ denote $d_j^\bfp$, the induced face operations from the disjoint clustering $\bfp$. We will use the fact that
$$
V(d^\slk_l(H))=V(d_l(H))=V(H)\smallsetminus V_l.
$$
Now
$$
\begin{array}{rcl}
d_i^\slk(d_j^{\slk}(H))&=&d_j^{\slk}(H)[V(d^\slk_j(H))\smallsetminus V_{i+1}]\cup G[\lk(V_{i+1})\cap V(d_j^\slk(H))]\\
&=&d_j^{\slk}(H)[V(d^\slk_j(H))\smallsetminus V_{i+1}]\cup G[\lk(V_{i+1})\cap (V(H)\smallsetminus V_j)]\\
&=&d_j^{\slk}(H)[V(d^\slk_j(H))\smallsetminus V_{i+1}]\cup G[\lk(V_{i+1})\cap (V(H)\smallsetminus V_j\smallsetminus V_{i+1})]\\
\end{array}
$$
because $\lk(V_{i+1})\cap V_{i+1}=\emptyset$.\\

The term $d_j^{\slk}(H)[V(d^\slk_j(H))\smallsetminus V_{i+1}]$ is the induced subgraph of $d_j^{\slk}(H)$ on its vertex subset
$$
V(d^\slk_j(H))\smallsetminus V_{i+1}=V(H)\smallsetminus V_j\smallsetminus V_{i+1}.
$$
By definition, $d_j^{\mathrm{lk}}(H)=H[V(H)\smallsetminus V_j(H)]\cup G[\lk(V_j(H))\cap V(H)]$. Restricting to the vertex subset
$$
W=V(H)\smallsetminus V_j\smallsetminus V_{i+1}
$$
we have
$$
d_j^{\slk}(H)[V(d^\slk_j(H))\smallsetminus V_{i+1}]=H[W]\cup G[\lk(V_j)\cap W]
$$
and so
$$
d_i^\slk(d_j^{\slk}(H))=H[W]\cup G[\lk(V_j)\cap W]\cup G[\lk(V_{i+1})\cap W].
$$
By the same arguments,
$$
d_j^\slk(d_{i+1}^{\slk}(H))=H[W]\cup G[\lk(V_j)\cap W]\cup G[\lk(V_{i+1})\cap W].
$$
\end{proof}
\subsubsection{Face Operations on Subgraphs with Marked Starting-Vertices}
Let $H$ be a subgraph of $G$ and let $S$ be a subset of the vertex set $V(H)$. When we remove $S$ from $H$, we wish to add as few edges as possible to make a blowup for the subgraph $H\smallsetminus S$ with the aim of preserving particular properties of $H$. For this construction, consider an extension of the working graph $G$ by adding an extra edge between any two distinct vertices $v,w$ in $G$ labeled as $\infty_{vw}$. This is an analogue to the idea of compactification in geometry. To showcase this we consider a special family of subgraphs.

Let $G$ be a directed/undirected (multi-)graph.
\begin{definition}
A \textit{subgraph with marked starting-vertices} of $G$ is a pair $(H, \SV(H))$ satisfying the following conditions:
\begin{enumerate}
\item[1)] $H$ is a subgraph of $G$,
\item[2)] $\SV(H)$ is a subset of $V(H)$,
\item[3)] every vertex $v$ of $H$ is reachable by a directed/undirected path out from a vertex in $\SV(H)$.
\end{enumerate}
\end{definition}

\noindent  In the case of digraphs, one may require further that there are no directed edges in $H$ incident into any vertex in $\SV(H)$. In this case, $\SV(H)$ acts as a source set for $H$. We want to give a description of face operations that is consistent across graphs and digraphs, therefore we do not require such an extra condition. There could be some redundant vertices contained in $\SV(H)$. A trivial example is to choose $\SV(H)=V(H)$, in which case there are no face operations. If $\SV(H)$ is a proper subset of $V(H)$, there will be nontrivial face operations.

Let $(H, \SV(H))$ be a subgraph with marked starting-vertices of $G$. We will now recursively construct a partition on the vertex set $V(H)$, called the \textit{neighborhood-extension partition},  in the following way. Let $V_0(H)=\SV(H)$ and suppose that $V_j(H)$ is constructed with $j\geq 0$. Define $V_{j+1}(H)$ as follows:
\begin{enumerate}
\item[1)] In the undirected case, let $V_{j+1}$ be the link set of the subgraph $H[V_j]$ in the graph $H$.
\item[2)] In the directed case,  let $V_{j+1}$ be the \textit{out-link set} of the subgraph $H[V_j]$ in the graph $H$. Here, for a (multi-)digraph $\Gamma$ and a subgraph $\Gamma'$, the \textit{closed out-neighborhood set} of $\Gamma'$ in $\Gamma$ is the union of $\Gamma'$ and the out-neighbors of $\Gamma'$ in $\Gamma$, denoted by $N^{\mathrm{out}}(\Gamma')$. The \textit{out-link set} of $\Gamma'$ is defined as
$$
\lk^{\mathrm{out}}(\Gamma')=N^{\mathrm{out}}(\Gamma')\smallsetminus V(\Gamma').
$$
For a subset $S$ of $V(\Gamma)$, let $N^{\mathrm{out}}(S)=N^{\mathrm{out}}(\Gamma[S])$ and $\lk^{\mathrm{out}}(S)=\lk^{\mathrm{out}}(\Gamma[S])$.
\end{enumerate}

If $(H, \SV(H))$ is a finite subgraph with marked starting-vertices of $G$, then above recursive construction will stop after finitely many steps, hence this gives a finite partition on $V(H)$.

To define face operations taking marked starting-vertices in to account, we embed $G$ into a larger graph $\hat G$, where $V(\hat G)=V(G)$ and $E(\hat G)$ is the extension of $E(G)$ by adding one edge $\infty_{vw}$ for $v\not=w\in V(G)$ in the undirected case, and by adding two directed edges $\infty_{vw}$ from $v$ to $w$ and $\infty_{wv}$ from $w$ to $v$ for vertices $v\not=w\in V(G)$ in the directed case.

Now let $(H, \SV(H))$ be a finite subgraph with marked starting-vertices of $\hat G$ with the neighborhood extension partition
$$
V(H)=\coprod_{i=0}^nV_i.
$$
We assume that $H\not=\emptyset$ and so  $V_0=\SV(H)\not=\emptyset$. From the recursive definition, $V_{i+1}\not=\emptyset$ implies that $V_i\not=\emptyset$. Therefore, $V_n$ is the last nonempty set in the recursive procedure.

We define the face operation $d^{\SV}_j$ on $H$, $0\leq j\leq n$, with $n>0$ as follows:
\begin{enumerate}
\item $d^{\SV}_0(H)=H[V(H)\smallsetminus V_0]$ with $SV(d^{\SV}_0(H))=V_1$.
\item $d^{\SV}_n(H)=H[V(H)\smallsetminus V_n]$ with $SV(d^{\SV}_n(H))=V_0$.
\item For $0\leq j<n$,
$$
d_j^{\SV}(H)=H[V(H)\smallsetminus V_j]\cup\mathcal{E}^H(V_{j-1},V_{j+1})
$$
with $\SV(d_j^{\SV}(H))=V_0$, where $\mathcal{E}^H(V_{j-1},V_{j+1})$ is a subset of the edge set $E(\hat G)$ consisting of $\infty_{vw}$ for $v\in V_{j-1}$ and $w\in V_{j+1}$ satisfying the property that there does not exist an edge from $v$ to $w$ in $H$\footnote{In the directed case, we only consider directed edges from a vertex in $V_{j-1}$ to a vertex in $V_{j+1}$ in $H$. If there are no directed edges in $H$ from $v\in V_{j-1}$ to $w\in V_{j+1}$, we add $\infty_{vw}$ to join them with direction from $v$ to $w$.}.
\end{enumerate}
Let $V_{-1}=V_{n+1}=\emptyset$. Then we can write in a unified way that
\begin{equation}
d_j^{\SV}(H)=H[V(H)\smallsetminus V_j]\cup\mathcal{E}^H(V_{j-1},V_{j+1})\textrm{ with } \SV(d_j^{\SV}(H))=V_{\delta_{j,0}}
\end{equation}
 for $0\leq j\leq n$, where $\delta_{a,b}$ is the Kronecker $\delta$ symbol.

We need to show that $(d^{\SV}_j(H), \SV(d^{\SV}_j(H)))$ is also a finite subgraph with marked starting-vertices of $\hat G$. This is straightforward because we add in $\infty_{vw}$ for possible missing edges in $H$ from $V_{j-1}$ to $V_{j+1}$. We also need to show that the $\Delta$-identity holds for these face operations. Let $i\geq j$. For $i>j$, we get 
$$
d_i^{\SV}(d_j^{\SV}(H))=d_j^{\SV}(d_{i+1}^{\SV}(H))=H[V(H)\smallsetminus V_j\smallsetminus V_{i+1}]\cup\mathcal{E}^H(V_{j-1},V_{j+1})\cup\mathcal{E}^H(V_i, V_{i+2}).
$$
For $i=j$, we have
$$
d_j^{\SV}(d_j^{\SV}(H))=d_j^{\SV}(d_{j+1}^{\SV}(H))=H[V(\H)\smallsetminus V_j\smallsetminus V_{j+1}]\cup \mathcal{E}^H(V_{j-1},V_{j+2}).
$$
This gives the following proposition.

\begin{proposition}\label{starting-vetex face identity}
In the set of finite subgraphs with marked starting-vertices of $\hat G$, the operations $d_j^{\SV}$ are well-defined and satisfy the $\Delta$-identity for face operations.

\hfill $\Box$
\end{proposition}

Now let $\H$ be a collection of finite subgraphs with marked starting-vertices of $G$. Under the extension $G\preceq \hat G$, $\H$ is also a collection of finite subgraphs with marked starting-vertices of $\hat G$. Let $X$ be the collection of finite subgraphs with marked starting-vertices of $\hat G$ given by all elements in $\H$ together with all of their iterated faces under face operations $d_j^{\SV}$. Then $X$ is $\Delta$-set dominated by $\hat G$. We call $(\mathcal{P}(G)\cap X, X)$ the \textit{super-hypergraph generated} by $\H$.

\subsubsection{Revisiting Path Complexes}

We now come back to the work of Yau's school on path (co-)homology of graphs~\cite{GLMY1}. Following their terminology, a simple digraph $G$ is a pair $(V,E)$, where $V$ is any set and $E\subseteq \{V\times V\smallsetminus \mathrm{diag}\}$. We will show that $d_j^{\SV}$ can be used to describe the face operations given in~\cite{GLMY1}. To do this, we consider the ``\textit{largest quotient simple digraph}'' $\tilde G$ on $\hat G$. Let $V(\tilde G)=V(\hat G)=V$, and for any ordered pair $(v, w)\in V\times V\smallsetminus \mathrm{diag}$, identify all the directed edges from $v$ to $w$ to give one such directed edge. Since $G$ is a simple graph, $\tilde G$ can be chosen to be a quotient of $\hat G$. Then $G$ is a subgraph of the simple digraph $\tilde G$. Let $v,w\in V(G)=V(\tilde G)$ be two distinct vertices. If there exists a directed edge $e_{vw}$ from $v$ to $w$, then $\infty_{vw}$ is identified with $e_{vw}$. Otherwise $\infty_{vw}$ is isolated. The graph $\tilde G$ can be considered as the completion of $G$ in the sense that it is the smallest complete simple digraph containing $G$. Here a complete simple digraph means a simple digraph with the property that for any two distinct vertices $v$ and $w$, there are exactly two directed edges with one from $v$ to $w$ and another from $w$ to $v$.

Replacing $\hat G$ by $\tilde G$, our face operation $d_j^{\SV}$ on paths in $\tilde G$ coincides with the face operation in~\cite[Section 4]{GLMY1} in the sense that it describes the $j$-th term in the boundary operators for chains and cochains on the path complex. Here, to match the definitions, a regular elementary path in~\cite[Section 4]{GLMY1} is a path in $\tilde G$ and an allowed regular elementary path in~\cite[Section 4]{GLMY1} is a path in the subgraph $G$. Then one can obtain the same objects by going via the definition of path homology in~\cite{GLMY1} or the definition of embedded homology of hypergraphs given above. In the undirected case, the operations $d_j^{\SV}$ on paths in $\tilde G$ coincide with the face operations in ~\cite[Section 5]{GLMY1}.

For directed multi-graphs (quivers), the operations $d_j^{\SV}$ describe the face operations in~\cite{GLMY6}. Similarly to the undirected case, we need to do a certain identification on $\hat G$. Following the argument in~\cite[Section 3]{GLMY6}, for a complete quiver $G$, $\infty_{vw}$ is identified with the $1$-chain given by the sum of all directed edges from $v$ to $w$ in $G$. This would define a chain complex on the path complex of a complete quiver. For an arbitrary quiver $G$, one can embed $G$ into its completion $\bar G$, and take the infimum chain complex (in Proposition~\ref{proposition9.2}) of the path complex of $G$ in the chain complex of the path complex of $\bar G$ to define path homology for the quiver $G$, see~\cite{GLMY6} for details.

As notions of paths and walks are commonly used in data analytics, generalising them to higher dimensional combinatorial objects, such as path complexes, could provide new tools for various applications. From a data science point of view, a graph $G$ is assumed as a working data. Then the  path homology gives some topological information on $G$. Using a scoring scheme of $G$, one could get persistent path homology of $G$, this gives a persistence diagram/barcode of $G$ as a topological feature. However, if the graphic data-set $G$ is large, the computational complexity may be an issue. From such a perspective, it would be reasonable to consider a selected sub complex, or more generally a selected graded subset of the path complex. The general theory developed in this article gives a framework that makes it possible to explore topological features from subcomplexes or graded subsets of path complexes.

\subsection{Descriptions of Simplicial Homology Via $\Delta$-sets}\label{subsection4.5}
In this subsetion we only consider mod $2$ homology, hence the coefficients are taken in $\F = \Z/2$.

\subsubsection{$\Delta$-neural network}
We can interpret a $\Delta$-set as a quiver or network. A similar object is a feed-forward neural network, as defined in ~\cite[Section 3.6.1]{Yadav}.

\textit{Feed-forward neural networks}, are the simplest form of artificial neural networks. The feed forward neural network was the first and arguably simplest type of artificial network devised. In this network, the information moves in only one direction, forward, from the input nodes, through the hidden nodes (if any) and to the output nodes. There are no cycles or loops in the network. In a feed-forward system, processing elements are arranged into distinct layers with each layer receiving input from the previous layer and outputting to the next layer.

Let $X=\{X_n\}_{n\geq0}$ be a $\Delta$-set. Each element in $X$ is considered as a node (vertex), that is, the node set $X$ is partitioned by layers labeled by $X_0,X_1,\ldots$. For each $x\in X_n$, assign one and only one arrow (directed edge) $d^n_{i,x}\colon x\to d_i(x)$ for $0\leq i\leq n$. This forms a set of arrows whose tails lie in $X_n$ and whose heads lie in $X_{n-1}$. So it forms a quiver. The following picture illustrates the arrows from $X_2$ to $X_1$.
$$
\xymatrix{
\bullet\ar@2{->}[d]^{d_2}_{d_0} \ar@{->}[drr]|{d_1}& \bullet\ar@{->}[dl]|{d_0}\ar@{->}[dr]|{d_2}\ar@{->}[drr]|{d_1} &   &  &\\
\bullet &  \bullet &\bullet &\bullet &\bullet\\}
$$
Rephrasing the definition of $\Delta$-set into the terminology of network,  a \textit{$\Delta$-neural network} is a quiver with distinct layers of nodes labeled by $X_0, X_1, \ldots$ such that for each node $x\in X_n$ with $n>0$, there are arrows $d^n_{i,x}$, $0\leq i\leq n$, tailed at $x$ and headed at some node $d^n_i(x)\in X_{n-1}$ such that
\begin{equation}\label{product-rule}
d^{n-1}_i(d^n_j(x))=d^{n-1}_j(d^n_{i+1}(x))
\end{equation}
for $0\leq j\leq i\leq n-1$. In a $\Delta$-neural network, the information only flows in one direction, from input nodes that could be located in different layers to the output nodes.

The adjacency relationship from the $n^{\text{th}}$ layer $X_n$  to the $n-1^{\text{th}}$ layer $X_{n-1}$ can be described by $(n+1)$ matrices as follows. For $x\in X_n$ and $y\in X_{n-1}$, let
$$
w^i_{x,y}=\left\{
\begin{array}{ll}
1&\textrm{ if } y=d_i(x)\\
0&\textrm{ otherwise}\\
\end{array}\right.
$$
for $0\leq i\leq n$. Let
$$
W_n(i)=(w^i_{x,y})_{x\in X_n, y\in X_{n-1}}
$$
be a $|X_n|\times |X_{n-1}|$ matrix, which is the matrix for the face operation $d_i\colon X_n\to X_{n-1}$. Equation~(\ref{product-rule}) can be rewritten as the following formula
\begin{equation}
W_n(j)W_{n-1}^t(i)=W_n(i+1)W_{n-1}^t(j)
\end{equation}
for $0\leq j\leq i\leq n-1$, where $A^t$ is the transpose of a matrix $A$.

Let $\H$ be a graded subset of the $\Delta$-set $X$. In our theory of super-hypergraphs, $\H$ carries part of the $\Delta$-set structure of $X$. More precisely, the face operation $d_i\colon X_n\to X_{n-1}$ induces a partially defined face operation $d_i\colon \H_n\rightsquigarrow \H_{n-1}$. By considering this as a network, $\H$ is full subnetwork of the $\Delta$-neural network induced by $X$, where a \textit{full subnetwork} is the induced network of the nodes of $\H$ in the neural network induced by $X$. Therefore, one can get homology on any full subnetwork of a $\Delta$-neural network using embedded homology of super-hypergraphs.

\begin{remark}
The product rule~(\ref{product-rule}) is important to define the boundary operator on the chains. But variations are possible. For example, one could vary the product rule for weighted simplicial complexes~\cite{Chengyuan}, or the boundary operators on cochains could be varied to account for twisted de Rham cohomology~\cite{Dwork}.

\end{remark}

\subsubsection{Descriptions of mod $2$ Homology}

We proceed by giving the ideas behind the intuition of mod $2$ homology of $\Delta$-sets and super-hypergraphs. Let $(\H,X)$ be a super-hypergraph. Then the $n$-chains on $X$ are linear combinations of the elements in $X_n$ with coefficients in $\Z/2$. So each $n$-chain $\alpha$ corresponds to a subset $\{x_1,\ldots,x_k\}\subseteq X_n$ given by $\alpha=x_1+x_2+\cdots+x_k$. Since the coefficients are in $\Z/2$,
$$
\partial(\alpha)=\sum_{i=1}^k\partial(x_i)=\sum_{i=1}^k\Sigma_{j=0}^n d_j(x_i)
$$
which is the \textit{trace} of the multi-subset $\{d_j(x_i)\ | \ 0\leq j\leq n, 1\leq i\leq k\}$ of $X_{n-1}$. Here the multiplicity of $y=d_j(x_i)\in X_{n-1}$ is the number of pairs $(j',i')$ such that $d_{j'}(x_{i'})=d_j(x_i)=y$, that is, the in-degree of the node $y\in X_{n-1}$ in the $\Delta$-neural network. Hence we have the following proposition.

\begin{proposition}
An $n$-chain $\alpha=x_1+x_2+\cdots+x_k$ is a mod $2$ cycle (that is, $\partial (\alpha)=0$) if and only if any node in the subset $$\{d_j(x_i)\ | \ 0\leq j\leq n, 1\leq i\leq k\}$$ of $X_{n-1}$ has even in-degree.\hfill $\Box$
\end{proposition}
This proposition indicates that one can consider the nodes in $X_{n-1}$ with even in-degrees in search for possible mod $2$ cycles in $n$-chains.

The following proposition follows from the fact that $Z_n(\inf_*^{C_*(X)}(\H))=\Z/2(\H_n)\cap Z_n(C_*(X))$.
\begin{proposition}
A mod $2$ cycle $\alpha=x_1+x_2+\cdots+x_k$ in the chains $C_n(X)$, with all $x_i$ distinct represents a cycle for the mod $2$ embedded homology $H_n^{\emb, X}(\H)$ if and only if $\{x_1,\ldots,x_k\}\subseteq \H_n$.\hfill $\Box$
\end{proposition}

An $n$-chain $\alpha=x_1+\cdots+x_k$ with all $x_i$ distinct is a boundary in the chain complex $C_*(X)$ if and only if the equation
$$
\alpha=\partial(\beta)
$$
where $\beta=y_1+\cdots +y_m$ with all $x_i$ distinct in $X_{n+1}$ has a solution. If there is a solution $\alpha=\partial(\beta)$, then $\{x_1,\ldots,x_k\}$ is the set of nodes in the multi-set $\{d_s(y_t) \ | \ 0\leq s\leq n+1,\ 1\leq t\leq m\}$ which have odd in-degrees.
This proves the following statement.
\begin{proposition}
An $n$-chain $\alpha=x_1+\cdots+x_k$ with all $x_i$ distinct is a boundary in the mod $2$ chain complex $C_*(X)$ if and only if there exists a subset $\{y_1,\ldots,y_m\}\subseteq X_{n+1}$ with $y_1,\ldots,y_m$ distinct such that $\{x_1,\ldots,x_k\}$ is the set of nodes in the multi-set $\{d_s(y_t) \ | \ 0\leq s\leq n+1,\ 1\leq t\leq m\}$ which have odd in-degrees.\hfill $\Box$
\end{proposition}

Note that $B_n(\inf_*^{C_*(X)}(\H))=\Z/2(\H_n)\cap\partial(\Z/2(\H_{n+1}))$.

\begin{proposition}
Let $\alpha=x_1+\cdots+x_k\in \Z/2(\H_n)$ with $x_1,\ldots,x_k$ distinct. Then $\alpha$ is a boundary in $\inf_*^{C_*(X)}(\H)$ if and only if there exists a subset $\{y_1,\ldots,y_m\}\subseteq \H_{n+1}$ with $y_1,\ldots,y_m$ distinct such that   $\{x_1,\ldots,x_k\}$ is the set of nodes in the multi-set $\{d_s(y_t) \ | \ 0\leq s\leq n+1,\ 1\leq t\leq m\}$ which have odd in-degrees.\hfill $\Box$
\end{proposition}

\section{Potential Applications}\label{sec:potent_apps}

\subsection{Potential Applications in Bio-molecular Structures and Drug Design}\label{subsec:potent_apps_biomed_drugdesgin}
Applications of persistent homology to molecular biology has achieved great success in computer aided drug design~\cite{Wei-drug-design2020,WangCangWei, Wei-patent}. According to~\cite[0005]{Wei-patent}, theoretical models for the study of the structure-function relationships of biomolecules are conventionally based on purely geometric techniques. Mathematically, these approaches make use of local geometric information such as: coordinates, distances, angles, areas and curvatures for the physical modeling of biomolecular systems. However, conventional purely geometry based models tend to be overwhelmed by too much structural detail and are frequently computationally intractable. Topological approaches to determining the nature of structure-function relationships of biomolecules provide a dramatic simplification compared to conventional geometry based approaches~\cite[0053]{Wei-patent}.

However, persistent homology neglects chemical and biological information during topological simplification and is thus not as competitive as geometry or physics-based alternatives in quantitative predictions~\cite{Wei-SIAM-news}. Element-specific persistent homology, or multi-component persistent homology built on colored biomolecular networks, has been introduced to retain chemical and biological information in topological abstractions~\cite{Wei2017-3}. This approach encodes biological properties---such as hydrogen bonds, van der Waals interactions, hydrophilicity, and hydrophobicity---into topological invariants, rendering a potentially revolutionary representation for biomolecules, according to the SIAM news~\cite{Wei-SIAM-news}.

\begin{figure}
	\centering
	\includegraphics[width=0.8\linewidth]{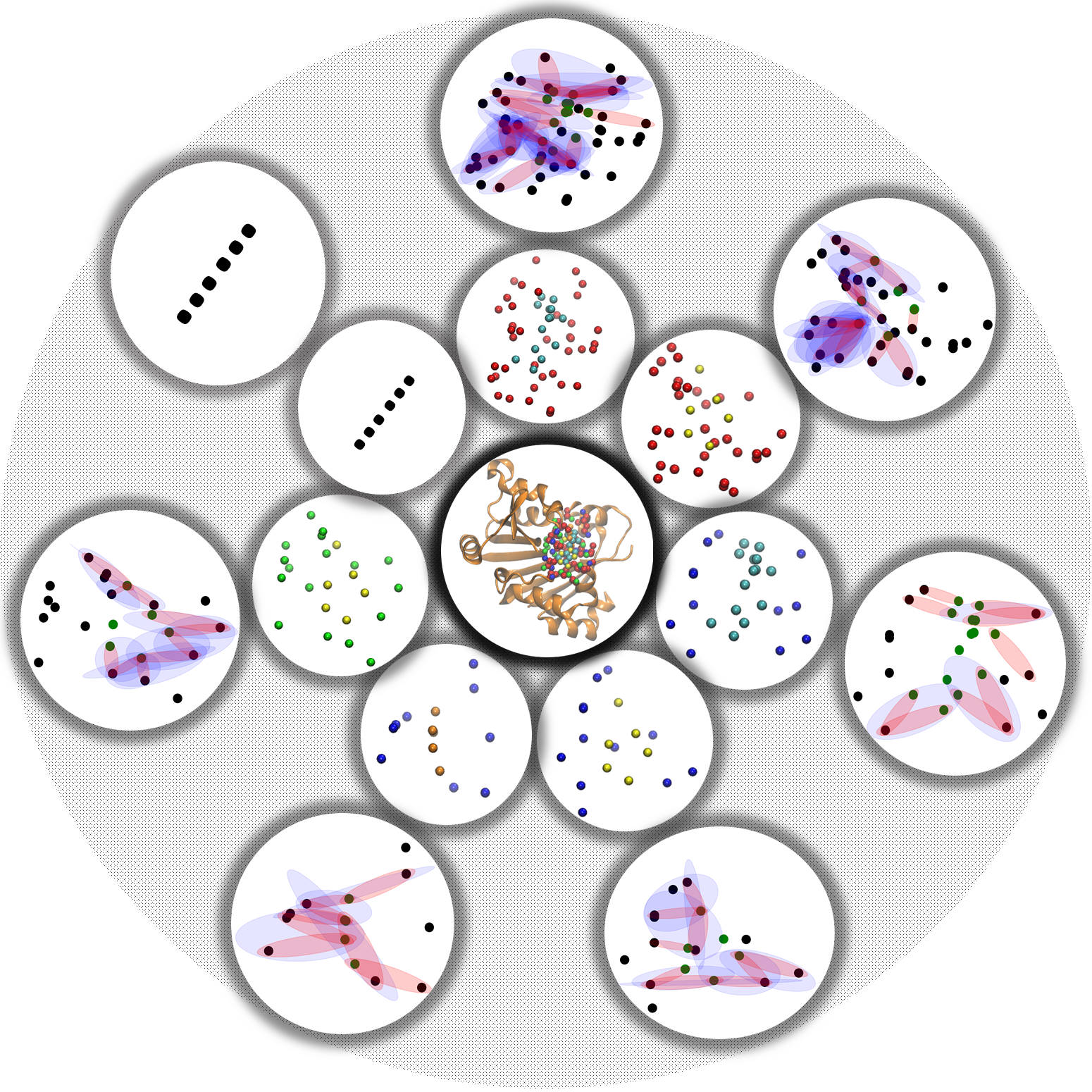}
	\caption{Illustration of an element-specific hypergraph model for a protein-ligand complex (ID 3PB3). The binding core region of the complex is decomposed into a series of element-specific atom-sets. The interactions between protein atom-sets and ligand atom-sets are modeled as a series of hypergraphs.}
	\label{fig:hypergraph}
\end{figure}

Recently, we have proposed hypergraph based persistent cohomology (HPC) for molecular representations in drug design~\cite{LWWX}. In our HPC model, the protein-ligand interactions at the molecular level are represented as a series of element-specific hypergraphs. Figure \ref{fig:hypergraph} illustrates our hypergraph model for a protein-ligand complex with ID 3PB3. Its binding core region is divided into a series of element-specific atom-sets. From these atom sets, element-specific hypergraphs can be constructed to characterize the interactions between protein atom-sets and ligand atom-sets at the level of atoms. Further, we have proposed a distance-related filtration process as illustrated in Figure \ref{fig:filtration}. With the embedded homology model for hypergraphs, we have developed the hypergraph persistent homology and cohomology for molecular characterization. Molecular features and descriptors can be obtained from hypergraph persistent barcodes and hypergraph enriched barcodes, and this information can be further combined with machine learning models, in particular, the gradient boosting tree (GBT). Our HPC-GBT model has performed well for protein-ligand binding affinity predictions. Its Pearson correlation coefficients (PCCs) for the three PDBbind datasets, including PDBbind-v2007, PDBbind-v2013 and PDBbind-v2016, are consistently better than traditional machine learning models with molecular descriptors.

\begin{figure}
	\centering
	\includegraphics[width=0.95\linewidth]{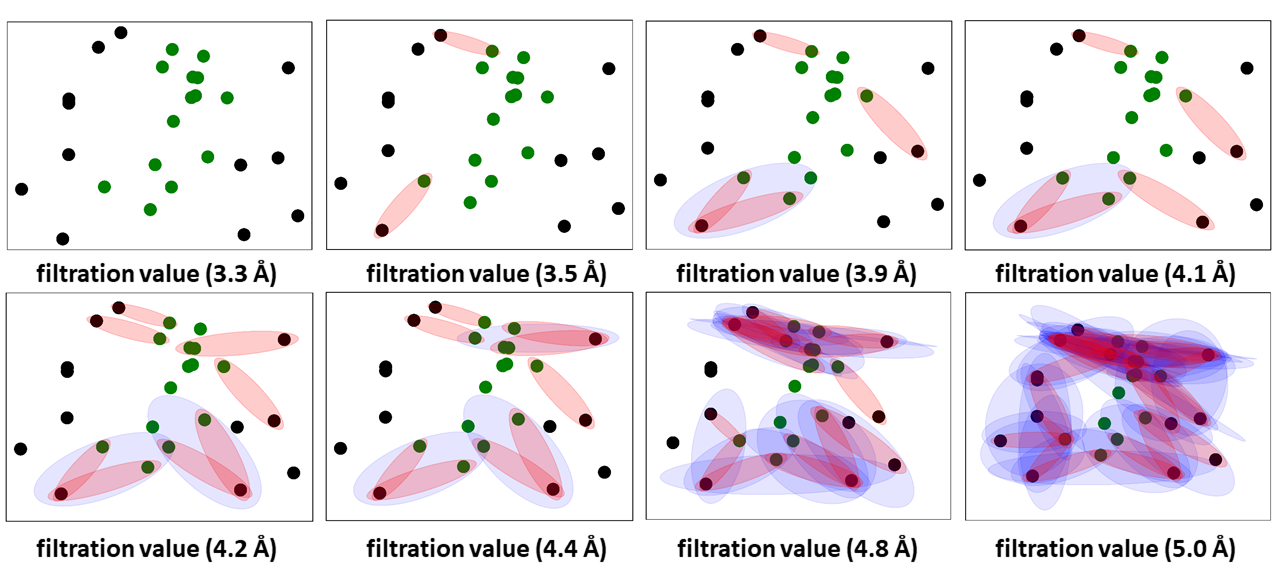}
	\caption{Illustration of a hypergraph-based filtration process for the protein-ligand complex with ID 3PB3.}
	\label{fig:filtration}
\end{figure}

A molecular representation based on super-hypergraphs could give more flexibility in molecular structure and interaction characterization. Unlike simplicies and hyperedges, super-hyperedges can incorporate local topological structures, that is, subgraphs. This provides a unique way to identify and describe molecular motifs, function groups, and domains. Further, boundary operators can be defined through vertex-deletion and edge-deletion, which provide ways to define different types of homology groups and thus characterize different types of inner topological connections. Moreover, different filtration processes can be defined by considering different scoring functions, which in turn will induce different super-hypergraph based persistent homology/cohomology. Finally, molecular descriptors/fingerprints can be generated from super-hypergraph models and further combined with machine learning models for molecular data analysis in materials, chemistry and biology.

\subsection{Potential Applications in Networks with Group Interactions}\label{subsec:potent_apps_group_inter}
The abstract of a recent review article~\cite{Battiston2020}, citing more than 800 references, reads, 
\begin{itemize}
\item[] 
\textit{The complexity of many biological, social and technological systems stems from the richness of the interactions among their units. Over the past decades, a variety of complex systems has been successfully described as networks whose interacting pairs of nodes are connected by links. Yet, from human communications to chemical reactions and ecological systems, interactions can often occur in groups of three or more nodes and cannot be described simply in terms of dyads...We review the measures designed to characterize the structure of these systems and the models proposed to generate synthetic structures, such as random and growing bipartite graphs, hypergraphs and simplicial complexes. We introduce the rapidly growing research on higher-order dynamical systems and dynamical topology, discussing the relations between higher-order interactions and collective behavior...}
\end{itemize}

Here we can see that simplicial complexes, a fundamental notion in algebraic topology, has been extensively used for providing representations of higher-order interactions~\cite[First paragraph of Section 2.1.3]{Battiston2020}. In some practical problems, the limitations of simplicial complexes due to completeness and vertax-determination present a problem. Hypergraphs provide a more general and unconstrained description of higher-order interactions~\cite[Paragraphs 2-4, page 7]{Battiston2020}.

Recent progress shows that simplicial homology can be naturally extended as a homology theory on hypergraphs~\cite{BLRW}. This provides topological invariants for geometric models using hypergraphs which have had successful applications in biomolecular structures and drug design, described in the previous subsection.

As an extension of hypergraphs, super-hypergraphs would provide more a general and unconstrained description of higher-order interations. If we assume that the higher-order interactions take place among the nodes in a working graph, which indicates the pre-existence of the pairwise bonds or the primary pairwise links between the nodes, then the most general and unconstrained description of higher-order interations would be a collection of finite subgraphs of the working graph, which is exactly the topic explored in this article. 

\section{Conclusion}

In this paper, we introduced a new mathematical theory which allows for topological invariants to be applied to broader range of problems, in particular enriching the methods of TDA. This new theory is suitable for both graph data and point cloud data analysis, while  overcoming various limitations of the standard persistent homology theory such as the topological noise and the constraining requirements to use data with metric. Using this new theory, the upgraded pipeline of TDA becomes indetermiinistic in nature allowing for flexibility and adjustments. Moreover, various new topological invariants can be constructed in our flexible setting. As highlighted in Subsections~\ref{ordinary persistent homology}\,-~\ref{subsection4.5}, based on this topological approach, more computational tools of algebraic topology will find applications in data science. For example, in algebraic topology the computation of simplicial homology of a space can be largely simplified by homotopically deforming it into a simpler shape, see~\cite{Hatcher}.

As each simplex of a simplicial complex is uniquely determined by its vertieces, simplicial complexes cannot model collections of subgraphs. To explore topological structures on space of subgraphs, in this paper we use $\Delta$-sets. Furthermore, we introduce the notion of sup-hypergraph, as a generalization of hypergraphs, which sets a stage for the exploration of  topological structures on subgraphs. The homology theory of super-hypergraph, established in Section~\ref{homology}, endows any collection of subgraphs with topological features.

In this work we also use the notion of scoring scheme. As highlighted in Section~\ref{section4}, scoring schemes are used to introduce persistence in an abstract setting without the use of any notion of metric. The classical constructions in persistent simplicial homology theory can be recovered using various scoring schemes.

We should point out that this work presents a theoretic research resulting in a framework that provides an upgraded topological approach to data science, with the aim
to foster further interactions between topology and data science.

Further research on super-hypergraphs is needed as this is a new and challenging mathematical concept. From a topological perspective there are many interesting questions to consider such as, the algebraic structure of the homology of these objects as well as the homotopy aspects of super-hypergraphs that are far less-understood. Furthermore, developments in the topological study of super-hypergraphs will feed into new innovative methods in TDA and other wide-ranging applications. Additionally, the computational complexity of the homology theory of super-hypergraphs is comparable to that of simplicial homology, therefore algorithms methods stemming from our approach will be similarly feasible as classical computation. 
\begin{center}\textmd{\textbf{Acknowledgements} }
\end{center}
This work was supported in part by Natural Science Foundation of China (NSFC
grant no. 11971144) and High-level Scientific Research Foundation of
Hebei Province. The third
author was supported  by Nanyang Technological University
Startup Grant M4081842 and Singapore Ministry of Education Academic
Research fund Tier 1 RG109/19, Tier 2 MOE2018-T2-1-033. The work of GWW was supported by NIH grant  GM126189, NSF grants DMS-1761320,   IIS-1900473, and DMS-2052983,   and NASA grant 80NSSC21M0023
\bibliographystyle{abbrv} 
\bibliography{Bibliography}

\end{document}